\documentclass{amsart}

\usepackage{amsmath}
\usepackage{amsfonts}
\usepackage{amssymb,enumerate}
\usepackage{amsthm}
\usepackage[all]{xy}
\usepackage{hyperref}

\newtheorem{lem}{Lemma}[section]
\newtheorem{cor}[lem]{Corollary}
\newtheorem{Fact}[lem]{Fact}
\newtheorem{prop}[lem]{Proposition}
\newtheorem{thm}[lem]{Theorem}

\newtheorem{Defn}[lem]{Definition}
\newtheorem{Notn}[lem]{Notation}
\newtheorem{Ex}[lem]{Example}
\newtheorem{Question}[lem]{Question}
\newtheorem{Property}[lem]{Property}
\newtheorem{Properties}[lem]{Properties}
\newtheorem{Discussion}[lem]{Remark}
\newtheorem{Construction}[lem]{Construction}
\newtheorem{Subprops}{}[lem]
\newtheorem{Para}[lem]{}

\newenvironment{defn}{\begin{Defn}\rm}{\end{Defn}}

\newenvironment{fact}{\begin{Fact}\rm}{\end{Fact}}

\newenvironment{para}{\begin{Para}\rm}{\end{Para}}
\newenvironment{disc}{\begin{Discussion}\rm}{\end{Discussion}}

\newtheorem{intthm}{Theorem}

\newcommand{\comp}[1]{\widehat{#1}}
\newcommand{\ideal}[1]{\mathfrak{#1}}
\newcommand{\m}{\ideal{m}}

\newcommand{\pd}{\operatorname{pd}}

\newcommand{\ext}{\operatorname{Ext}}

\newcommand{\bbz}{\mathbb{Z}}
\newcommand{\bbq}{\mathbb{Q}}

\newcommand{\id}{\operatorname{id}}

\newcommand{\HH}{\operatorname{H}}
\newcommand{\Hom}{\operatorname{Hom}}

\newcommand{\coker}{\operatorname{Coker}}
\newcommand{\fd}{\operatorname{fd}}

\newcommand{\spec}{\operatorname{Spec}}

\newcommand{\tor}{\operatorname{Tor}}
\newcommand{\vf}{\varphi}

\newcommand{\im}{\operatorname{Im}}

\newcommand{\xra}{\xrightarrow}

\newcommand{\shift}{\mathsf{\Sigma}}

\newcommand{\onto}{\twoheadrightarrow}

\newcommand{\Ker}{\operatorname{Ker}}

\newcommand{\wti}{\widetilde}

\newcommand{\cat}[1]{\mathcal{#1}}
\newcommand{\catx}{\cat{X}}
\newcommand{\caty}{\cat{Y}}
\newcommand{\catm}{\cat{M}(R)}
\newcommand{\catv}{\cat{V}}
\newcommand{\catw}{\cat{W}}
\newcommand{\catg}{\cat{G}}
\newcommand{\catp}{\cat{P}(R)}
\newcommand{\catf}{\cat{F}(R)}
\newcommand{\cati}{\cat{I}(R)}
\newcommand{\catpp}{\cat{P}}
\newcommand{\catff}{\cat{F}}
\newcommand{\catii}{\cat{I}}

\newcommand{\catgii}{\cat{GI}}
\newcommand{\catgi}{\cat{GI}(R)}

\newcommand{\catgf}{\cat{GF}(R)}
\newcommand{\catgff}{\cat{GF}}
\newcommand{\catgfc}{\cat{GF}_C(R)}
\newcommand{\catgfr}{\cat{GF}_R(R)}
\newcommand{\catgfcc}{\cat{GF}_C}
\newcommand{\catgic}{\cat{GI}_C(R)}
\newcommand{\catgir}{\cat{GI}_R(R)}

\newcommand{\catgpc}{\cat{GP}_C(R)}

\newcommand{\catac}{\cat{A}_C(R)}
\newcommand{\catbc}{\cat{B}_C(R)}
\newcommand{\catic}{\cat{I}_C(R)}
\newcommand{\catpc}{\cat{P}_C(R)}
\newcommand{\catfc}{\cat{F}_C(R)}
\newcommand{\caticc}{\cat{I}_C}
\newcommand{\catpcc}{\cat{P}_C}
\newcommand{\catfcc}{\cat{F}_C}
\newcommand{\catfcot}{\cat{F}^{\text{cot}}(R)}
\newcommand{\catfcott}{\cat{F}^{\text{cot}}}

\newcommand{\catfccot}{\cat{F}_C^{\text{cot}}(R)}
\newcommand{\catfccott}{\cat{F}_C^{\text{cot}}}
\newcommand{\cathc}{\cat{H}_C}

\newcommand{\PP}{\catpp}

\newcommand{\catpd}[1]{\cat{#1}\text{-}\pd}
\newcommand{\xpd}{\catpd{X}}

\newcommand{\catid}[1]{\cat{#1}\text{-}\id}
\newcommand{\yid}{\catid{Y}}

\newcommand{\gfcpd}{\cat{GF}_C\text{-}\pd}
\newcommand{\gfd}{\operatorname{Gfd}}

\newcommand{\icid}{\cat{I}_C\text{-}\id}
\newcommand{\fcpd}{\cat{F}_C\text{-}\pd}
\newcommand{\fcotpd}{\cat{F}^{\text{cot}}\text{-}\pd}
\newcommand{\fccotpd}{\cat{F}^{\text{cot}}_C\text{-}\pd}

\newcommand{\finrescat}[1]{\operatorname{res}\comp{\cat{#1}}}

\newcommand{\finrescatx}{\finrescat{X}}
\newcommand{\finrescaty}{\finrescat{Y}}

\newcommand{\fincorescat}[1]{\operatorname{cores}\comp{\cat{#1}}}
\newcommand{\propcorescat}[1]{\operatorname{cores}\wti{\cat{#1}}}
\newcommand{\fincorescatx}{\fincorescat{X}}

\newcommand{\propcorescatfccot}{\operatorname{cores}\wti{\catfccot}}
\newcommand{\finrescatfccot}{\operatorname{res}\comp{\catfccot}}
\newcommand{\finrescatfc}{\operatorname{res}\comp{\catfc}}
\newcommand{\finrescatpc}{\operatorname{res}\comp{\catpc}}
\newcommand{\finrescatfccott}{\operatorname{res}\comp{\catfccott}}
\newcommand{\finrescatfcot}{\operatorname{res}\comp{\catfcot}}
\newcommand{\finrescatgfc}{\operatorname{res}\comp{\catgfc}}

\newcommand{\fincorescatfccot}{\operatorname{cores}\comp{\catfccot}}

\newcommand{\fincorescaty}{\fincorescat{Y}}

\newcommand{\propcorescaty}{\propcorescat{Y}}

\newcommand{\propcorescatw}{\propcorescat{W}}

\renewcommand{\geq}{\geqslant}
\renewcommand{\leq}{\leqslant}
\renewcommand{\ker}{\Ker}

\newcommand{\G}{\mathcal{G}}

\newcommand{\gpcpd}{\operatorname{\G\PP_C\text{-}\pd}}

\renewcommand{\hom}{\Hom}

\newcommand{\qmodz}{\bbq/\bbz}
\newcommand{\pdual}[1]{\hom_{\bbz}(#1,\qmodz)}

\begin{document}

\bibliographystyle{amsplain}

\author{Sean Sather-Wagstaff}

\address{Sean Sather-Wagstaff, Department of Mathematics,
300 Minard Hall,
North Dakota State University,
Fargo, North Dakota 58105-5075, 
USA}
\email{Sean.Sather-Wagstaff@ndsu.edu}
\urladdr{http://math.ndsu.nodak.edu/faculty/ssatherw/}

\author{Tirdad Sharif}
\address{Tirdad Sharif, School of Mathematics, Institute for Studies in
Theoretical Physics and Mathematics, P. O. Box 19395-5746, Tehran, Iran}
\email{sharif@ipm.ir}
\urladdr{http://www.ipm.ac.ir/IPM/people/personalinfo.jsp?PeopleCode=IP0400060}
\thanks{TS is supported by a grant from IPM, (No. 83130311).}

\author{Diana White}
\address{Diana White, Department of Mathematics,
LeConte College,
1523 Greene Street,
University of South Carolina,
Columbia, SC 29208,
USA}
\email{dwhite@math.sc.edu}
\urladdr{http://www.math.sc.edu/~dwhite/}

\title[AB-Contexts and Stability
for Gorenstein Flat Modules]{AB-Contexts and Stability
for Gorenstein Flat Modules with Respect to
Semidualizing Modules}



\begin{abstract}
We investigate the properties of categories of
$\text{G}_C$-flat $R$-modules where $C$ is a
semidualizing module over a commutative noetherian ring $R$.
We prove that the category of all
$\text{G}_C$-flat $R$-modules is part of a weak AB-context,
in the terminology of Hashimoto.  In particular,
this allows us to deduce the existence of certain
Auslander-Buchweitz approximations
for $R$-modules of finite $\text{G}_C$-flat dimension.
We also prove that two procedures for building $R$-modules
from complete resolutions by certain 
subcategories of $\text{G}_C$-flat $R$-modules
yield only the modules in the original subcategories.
\end{abstract}

\maketitle

\section*{Introduction}

Auslander
and Bridger~\cite{auslander:adgeteac, auslander:smt} 
introduce the modules of finite G-dimension 
over a commutative noetherian ring $R$, in part, to identify
a class of finitely generated $R$-modules with particularly nice duality properties
with respect to $R$.  
They are exactly the $R$-modules  which admit 
a finite resolution by modules of G-dimension 0.
As a special case, the duality theory for these modules recovers the
well-known duality theory for  finitely generated modules
over a Gorenstein ring.

This notion has been extended
in several directions. For instance, 
Enochs, Jenda and Torrecillas~\cite{enochs:gipm, enochs:gf}
introduce the Gorenstein projective modules and the
Gorenstein flat modules; these are analogues of modules of  G-dimension 0
for the  non-finitely generated arena.
Foxby~\cite{foxby:gmarm},
Golod~\cite{golod:gdagpi}
and Vasconcelos~\cite{vasconcelos:dtmc}
focus on finitely generated modules, but consider duality with
respect to a semidualizing module $C$.
Recently,
Holm and J\o rgensen~\cite{holm:smarghd} have unified these approaches
with the $\text{G}_C$-projective modules and the
$\text{G}_C$-flat modules.
For background and definitions, see Sections~\ref{sec01} and~\ref{sec02}.

The purpose of this paper is to use 
cotorsion flat modules 
in order to further study the $\text{G}_C$-flat modules, which 
are more technically challenging to investigate than the
$\text{G}_C$-projective modules. 
Cotorsion flat modules 
have been successfully used to investigate  flat modules,
for instance in the work of Xu~\cite{xu:fcm}, and 
this paper shows how they 
are similarly well-suited for studying the $\text{G}_C$-flat modules. 

More specifically, an $R$-module is \emph{$C$-flat $C$-cotorsion}  
when is isomorphic to an $R$-module of the form $F\otimes_R C$ where $F$ is flat and cotorsion.
We let $\catfccot$ denote the category of all 
$C$-flat $C$-cotorsion $R$-modules, and we let
$\finrescatfccot$ denote the category of all $R$-modules admitting a finite
resolution by
$C$-flat $C$-cotorsion $R$-modules.
The first step of our analysis is carried out in Section~\ref{sec04}
where we investigate the fundamental properties of these categories;
see Theorem~\ref{thma}\eqref{thma2} for some of the conclusions from this 
section.

Section~\ref{sec03} contains our analysis of the category of
$\text{G}_C$-flat modules, denoted $\catgfc$.  This section
culminates in the following theorem. In the terminology
of Hashimoto~\cite{hashimoto:abaem}, it says that the triple
$(\catgfc,\finrescatfccot,\catfccot)$ 
satisfies the axioms for a weak AB-context.
The proof of this result is in~\eqref{thm0301}.

\begin{intthm} \label{thma}
Let $C$ be a semidualizing $R$-module.
\begin{enumerate}[\quad\rm(a)]
\item \label{thma1}
$\catgfc$ is closed under extensions, kernels of epimorphisms
and  summands.
\item \label{thma2}
$\finrescatfccot$ is closed under cokernels of monomorphisms,
extensions and  summands, and $\finrescatfccot\subseteq\finrescatgfc$.
\item \label{thma3}
$\catfccot= \catgfc\cap \finrescatfccot$, and 
$\catfccot$ is an injective cogenerator for $\catgfc$.
\end{enumerate}
\end{intthm}

In conjunction with~\cite[(1.12.10)]{hashimoto:abaem},
this result implies 
many of the conclusions of~\cite{auslander:htmcma}
for the triple
$(\catgfc,\finrescatfccot,\catfccot)$.
For instance, we conclude that every module $M$ of
finite $\text{G}_C$-flat dimension 
fits in an exact sequence
$$0\to Y\to X\to M\to 0$$
such that $X$ is in $\catgfc$ and $Y$ is in $\finrescatfccot$.
Such ``approximations'' have been very useful, for instance, in the
study of modules of finite G-dimension.
See Corollary~\ref{cor0301} for this and other conclusions.

In Section~\ref{sec05} we apply these techniques to continue our
study of stability properties of Gorenstein categories, initiated
in~\cite{sather:sgc}. 
For each subcategory
$\catx$ of the category of $R$-modules,
let $\catg^1(\catx)$ denote the  category of all $R$-modules
isomorphic to $\coker(\partial^X_1)$ for some
exact complex $X$ in $\catx$ such that the complexes $\hom_R(X',X)$ and
$\hom_R(X,X')$ are exact for each module $X'$ in $\catx$.  
This definition is a modification of the construction of $\text{G}_C$-projective $R$-modules.
Inductively, set $\catg^{n+1}(\catx)=\catg(\catg^n(\catx))$
for each $n\geq 1$.  The techniques of this paper allow us to prove the
following $\text{G}_C$-flat versions of some results of~\cite{sather:sgc};
see Corollary~\ref{cor0501} and Theorem~\ref{prop0702}.

\begin{intthm} \label{thmb}
Let $C$ be a semidualizing $R$-module and let $n\geq 1$.
\begin{enumerate}[\quad\rm(a)]
\item \label{thmb1}
We have $\catg^n(\catgfc\cap\catbc)=\catgfc\cap\catbc$.
\item \label{thmb2}
If
$\dim(R)<\infty$, then
$\catg^n(\catfccot)=\catgfc\cap\catbc\cap\catfc^{\perp}$.
\end{enumerate}
\end{intthm}

Here $\catbc$ is the Bass class associated to $C$, and $\catfc^{\perp}$
is the category of all $R$-modules $N$ such that
$\ext^{\geq 1}_R(F\otimes_R C,N)=0$ for each flat $R$-module $F$.
In particular, when $C=R$ this result yields
$\catg^n(\catgf)=\catgf$ and,
when $\dim(R)$ is finite, $\catg^n(\catfcot)=\catgf\cap\catf^{\perp}$.

\section{Modules, Complexes and Resolutions}\label{sec01}

We begin with some notation and terminology for use throughout this paper.

\begin{defn} \label{notation01}
Throughout this work
$R$ is a commutative noetherian ring and
$\catm$ is the category of $R$-modules.
We use the term ``subcategory'' to mean a ``full, additive subcategory
$\catx\subseteq\catm$
such that, for all $R$-modules $M$ and $N$, if $M\cong N$ and
$M\in\catx$, then $N\in\catx$.'' 
Write $\catp$, $\catf$ and $\cati$
for the subcategories of projective, flat and injective
$R$-modules, respectively.
\end{defn}

\begin{defn} \label{notation01a}
We fix subcategories $\catx$, $\caty$, $\catw$, and $\catv$  of $\catm$ such that
$\catw\subseteq\catx$
and $\catv \subseteq\caty$.
Write $\catx\perp\caty$
if $\ext_R^{\geq1}(X,Y)=0$ for each $X\in \catx$ and each $Y\in\caty$.
For an $R$-module $M$, write $M\perp\caty$ (resp., $\catx\perp M$)
if $\ext_R^{\geq1}(M,Y)=0$ for each  $Y\in \caty$
(resp., if $\ext_R^{\geq1}(X,M)=0$ for each  $X\in \catx$).
Set 
$$\catx^{\perp}=\text{the subcategory of 
$R$-modules $M$ such that $\catx\perp M$.}$$
We say $\catw$ is a \emph{cogenerator} for $\catx$ if,
for each  $X\in\catx$, there is an exact sequence 
$$0\to X\to W\to X'\to 0$$
such that $W\in \catw$ and $X'\in \catx$;
and
$\catw$ is an \emph{injective cogenerator} for $\catx$ if
$\catw$ is a cogenerator for $\catx$ and $\catx\perp\catw$.
The terms \emph{generator} and \emph{projective generator}
are defined dually.

We say that $\catx$ is \emph{closed under extensions} when, for every
exact sequence
\begin{equation} \tag{$\ast$} \label{eq01a}
0\to M'\to M\to M''\to 0
\end{equation}
if $M',M''\in \catx$, then $M\in\catx$.
We say that $\catx$ is \emph{closed under kernels of monomorphisms} when, for every
exact sequence~\eqref{eq01a},
if $M',M\in \catx$, then $M''\in\catx$.
We say that $\catx$ is \emph{closed under cokernels of epimorphisms} when, for every
exact sequence~\eqref{eq01a},
if $M,M''\in \catx$, then $M'\in\catx$.
We say that $\catx$ is \emph{closed under summands} when, for every
exact sequence~\eqref{eq01a},
if $M\in \catx$ and~\eqref{eq01a} splits, then $M',M''\in\catx$.
We say that $\catx$ is \emph{closed under products} when, for every
set $\{M_{\lambda}\}_{\lambda\in\Lambda}$ of modules in $\catx$,
we have $\prod_{\lambda\in\Lambda}M_{\lambda}\in\catx$.
\end{defn}

\begin{defn} \label{notation07}
We employ the notation  from~\cite{christensen:gd} for $R$-complexes.
In particular, $R$-complexes are indexed homologically
$$M =\cdots\xra{\partial^M_{n+1}}M_n\xra{\partial^M_n}
M_{n-1}\xra{\partial^M_{n-1}}\cdots$$
with
$n$th homology module denoted
$\HH_n(M)$.
We frequently identify $R$-modules with $R$-complexes concentrated in degree 0.

Let $M,N$ be $R$-complexes.
For each integer $i$,
let $\shift^i M$ denote the complex with
$(\shift^i M)_n=M_{n-i}$ and $\partial_n^{\shift^i M}=(-1)^i\partial_{n-i}^M$.
Let $\hom_R(M,N)$ and $M\otimes_R N$ denote the associated 
Hom complex and tensor product complex, respectively.
A morphism $\alpha\colon M\to N$ is a \emph{quasiisomorphism} 
when each induced map $\HH_n(\alpha)\colon\HH_n(M)\to\HH_n(N)$
is bijective. Quasiisomorphisms are designated by the symbol $\simeq$.

The complex $M$ is \emph{$\hom_R(\catx,-)$-exact} if the complex
$\hom_R(X,M)$ is exact for each $X \in\catx$.
Dually, the complex $M$ is \emph{$\hom_R(-,\catx)$-exact} if 
$\hom_R(M,X)$ is exact for each $X\in \catx$,
and $M$ is \emph{$-\otimes_R\catx$-exact} if 
$M\otimes_R X$ is exact for each $X \in \catx$.
\end{defn}

\begin{defn} \label{notation03}
When
$X_{-n}=0=\HH_n(X)$ for all $n>0$, the natural morphism
$X\to\HH_0(X)=M$ is a quasiisomorphism, that is,
the following sequence is exact
$$X^+ = \cdots\xra{\partial^X_{2}}X_1
\xra{\partial^X_{1}}X_0\to M\to 0.$$ 
In this event, $X$ is an
\emph{$\catx$-resolution} of $M$ if each $X_n$ is  in $\catx$, and
$X^+$ is the \emph{augmented
$\catx$-resolution} of $M$ associated to $X$.
We write ``projective resolution'' in lieu of
``$\cat{P}$-resolution'', and we write ``flat resolution'' in lieu of
``$\cat{F}$-resolution''.
The \emph{$\catx$-projective dimension} of $M$ is the quantity
$$\xpd_R(M)=\inf\{\sup\{n\geq 0\mid X_n\neq 0\}\mid \text{$X$ is an
$\catx$-resolution of $M$}\}.$$
The modules of $\catx$-projective dimension 0 are
the nonzero modules of $\catx$.
We set
$$\finrescatx=
\text{the subcategory of $R$-modules $M$  with $\xpd_R(M)<\infty$.}$$
One checks easily that $\finrescatx$ is additive and contains $\catx$.
Following establised conventions, we set
$\pd_R(M)=\catpd{P}_R(M)$ and $\fd_R(M)=\catpd{F}_R(M)$.

The term \emph{$\caty$-coresolution} is defined dually.
The \emph{$\caty$-injective dimension} of $M$ is denoted $\yid_R(M)$,
and the \emph{augmented
$\caty$-coresolution} associated to a $\caty$-coresolution $Y$ is denoted $^+Y$.
We write ``injective resolution'' for
``$\cat{I}$-coresolution'', and we set
$$\fincorescaty=
\text{the subcategory of $R$-modules $N$ with $\yid_R(N)<\infty$}$$
which is additive and contains $\caty$.
\end{defn}

\begin{defn} \label{notation05}
A $\caty$-coresolution $Y$ is \emph{$\catx$-proper} if the
the augmented resolution $^+Y$ is $\hom_R(,-\catx)$-exact. We set
$$\propcorescaty=
\text{the subcategory of $R$-modules  admitting a $\caty$-proper
$\caty$-coresolution.}$$
One checks readily that $\propcorescaty$ is additive and contains
$\caty$. The term
\emph{$\caty$-proper $\catx$-resolution} is defined dually.
\end{defn}

\begin{defn} \label{defn0205}
An \emph{$\catx$-precover} of an $R$-module $M$ is an $R$-module homomorphism
$\vf\colon X\to M$ where $X\in\catx$ such that, for each $X'\in X$, the homomorphism
$\Hom_R(X',\vf)\colon\hom_R(X',X)\to\hom_R(X',M)$ is surjective.
An $\catx$-precover $\vf\colon X\to M$ is an \emph{$\catx$-cover}
if,  every endomorphism $f\colon X\to X$ such that $\vf=\vf f$
is an automorphism.
The terms \emph{preenvelope} and \emph{envelope}
are defined dually.
\end{defn}

The next three lemmata  have standard proofs;
see~\cite[proofs of (2.1) and (2.3)]{auslander:htmcma}.

\begin{lem} \label{perp03}
Let $0\to M_1\to M_2\to M_3\to 0$ be an exact sequence of $R$-modules.
\begin{enumerate}[\quad\rm(a)]
\item \label{perp03item1}
If $M_3\perp\catw$, then $M_1\perp\catw$ if and only if $M_2\perp\catw$.
If $M_1\perp\catw$  and  $M_2\perp\catw$,
then $M_3\perp\catw$
if and only if the given sequence is  $\hom_R(-,\catw)$-exact.
\item \label{perp03item2}
If $\catv\perp M_1$, then $\catv\perp M_2$ if and only if $\catv\perp M_3$.
If $\catv\perp M_2$  and  $\catv\perp M_3$,
then $\catv\perp M_1$
if and only if the given sequence is  $\hom_R(\catv,-)$-exact.
\item \label{perp03item3}
If $\tor^R_{\geq 1}(M_3, \catv)=0$, then $\tor^R_{\geq 1}(M_1, \catv)=0$ if and only if 
$\tor^R_{\geq 1}(M_2, \catv)=0$.
If $\tor^R_{\geq 1}(M_1, \catv)=0=\tor^R_{\geq 1}(M_2, \catv)$,
then $\tor^R_{\geq 1}(M_3, \catv)=0$
if and only if the given sequence is  $-\otimes_R\catv$-exact.
\qed
\end{enumerate}
\end{lem}

\begin{lem} \label{gencat01}
If $\catx\perp\caty$, then  $\catx\perp\finrescaty$ and $\fincorescatx\perp\caty$.
\qed
\end{lem}

\begin{lem} \label{lem0216}
Let $X$ be an exact $R$-complex.
\begin{enumerate}[\quad\rm(a)]
\item \label{lem0216a}
Assume $X_i\perp\catv$ for all $i$. If $X$ is $\hom_R(-,\catv)$-exact,
then $\ker(\partial^X_i)\perp\catv$ for all $i$.
Conversely, if $\ker(\partial^X_i)\perp\catv$ for all $i$
or if $X_i=0$ for all $i\ll 0$, then $X$ is $\hom_R(-,\catv)$-exact.
\item \label{lem0216b}
Assume $\catv \perp X_i$ for all $i$. If $X$ is $\hom_R(\catv,-)$-exact,
then $\catv\perp\ker(\partial^X_i)$ for all $i$.
Conversely, if $\catv\perp\ker(\partial^X_i)$ for all $i$
or if $X_i=0$ for all $i\gg 0$, then $X$ is $\hom_R(\catv,-)$-exact.
\item \label{lem0216c}
Assume $\tor^R_{\geq 1}(X_i,\catv)=0$ for all $i$. If the complex
$X$ is $-\otimes_R\catv$-exact,
then $\tor^R_{\geq 1}(\ker(\partial^X_i),\catv)=0$ for all $i$.
Conversely, if $\tor^R_{\geq 1}(\ker(\partial^X_i),\catv)=0$ for all $i$
or if $X_i=0$ for all $i\ll 0$, then $X$ is $-\otimes_R\catv$-exact.
\qed
\end{enumerate}
\end{lem}

A careful reading of the proofs of~\cite[(2.1),(2.2)]{sather:sgc}
yields the next  result.

\begin{lem} \label{lem0602}
Assume that $\catw$ is an injective cogenerator for $\catx$.
If $M$ has an $\catx$-coresolution that is $\catw$-proper and $M\perp\catw$, then
$M$ is in $\propcorescatw$.
\qed
\end{lem}

\section{Categories of Interest}\label{sec02}

This section contains definitions of
and basic facts about the categories to be investigated in this paper.

\begin{defn} \label{defn0203}
An $R$-module $M$ is \emph{cotorsion} if $\catf\perp M$.
We set
$$
\catfcot=\text{the subcategory of  flat cotorsion $R$-modules.}
$$
\end{defn}

\begin{defn} \label{defn0204}
The \emph{Pontryagin dual} or \emph{character module}
of an $R$-module $M$ is the $R$-module
$M^*=\Hom_{\bbz}(M,\bbq/\bbz)$.
\end{defn}

One implication in the following lemma is from~\cite[(3.1.4)]{xu:fcm},
and the others are established similarly.

\begin{lem} \label{lem0202}
Let $M$ be an $R$-module.
\begin{enumerate}[\quad\rm(a)]
\item \label{lem0202a}
The Pontryagin dual $M^*$ is $R$-flat if and only if $M$ is $R$-injective.
\item \label{lem0202b}
The Pontryagin dual $M^*$ is $R$-injective if and only if $M$ is $R$-flat. 
\qed
\end{enumerate}
\end{lem} 

Semidualizing modules, defined next, form the basis for our categories of interest.

\begin{defn} \label{defn0201}
A finitely generated $R$-module $C$ is \emph{semidualzing} 
if 
the natural homothety morphism
$R\to \Hom_R(C,C)$ is an isomorphism and
$\ext^{\geq 1}_R(C,C)=0$.
An $R$-module $D$ is \emph{dualizing} if it is 
semidualizing and has finite injective dimension.

Let $C$ be a semidualizing $R$-module.
We set
\begin{align*}
\catpc&=\text{the subcategory of modules $P\otimes_R C$ where $P$ is $R$-projective}\\
\catfc&=\text{the subcategory of modules $F\otimes_R C$ where $F$ is $R$-flat}\\
\catfccot&=\text{the subcategory of modules $F\otimes_R C$ where $F$ is flat and cotorsion} \\
\catic&=\text{the subcategory of modules $\Hom_R(C,I)$ where $I$ is $R$-injective.}
\end{align*}
Modules in $\catpc$, $\catfc$, $\catfccot$ and $\catic$ are called \emph{$C$-projective},
\emph{$C$-flat}, \emph{$C$-flat $C$-cotorsion},
and \emph{$C$-injective}, respectively.
An $R$-module $M$ is \emph{$C$-cotorsion} if $\catfc\perp M$.
\end{defn}

\begin{disc} \label{defn0401}
We justify the terminology ``$C$-flat $C$-cotorsion'' in Lemma~\ref{lem0201}
where we show that $M$ is $C$-flat $C$-cotorsion if and only if it is
$C$-flat and $C$-cotorsion.
\end{disc}

The following categories
were introduced  by Foxby~\cite{foxby:gdcmr},
Avramov and Foxby~\cite{avramov:rhafgd},
and Christensen~\cite{christensen:scatac}, thought the idea 
goes at least back to Vasconcelos~\cite{vasconcelos:dtmc}.

\begin{defn} \label{notation08d}
Let $C$ be a
semidualizing $R$-module.  
The \emph{Auslander class} of $C$ is the subcategory $\catac$
of $R$-modules $M$ such that 
\begin{enumerate}[\quad(1)]
\item $\tor^R_{\geq 1}(C,M)=0=\ext_R^{\geq 1}(C,C\otimes_R M)$, and
\item The natural map $M\to\Hom_R(C,C\otimes_R M)$ is an isomorphism.
\end{enumerate}
The \emph{Bass class} of $C$ is the subcategory $\catbc$
of $R$-modules $M$ such that 
\begin{enumerate}[\quad(1)]
\item $\ext_R^{\geq 1}(C,M)=0=\tor^R_{\geq 1}(C,\Hom_R(C,M))$, and 
\item The natural evaluation map $C\otimes_R\Hom_R(C,M)\to M$ is an isomorphism.
\end{enumerate}
\end{defn}

\begin{fact}\label{projac}
Let $C$ be a
semidualizing $R$-module.  
The categories $\catac$ and $\catbc$ are closed under extensions,
kernels of epimorphisms and cokernels of monomorphism; see~\cite[Cor.\ 6.3]{holm:fear}.
The category $\catac$ contains all modules of finite
flat dimension and those of finite 
$\caticc$-injective dimension, and the category $\catbc$ contains all modules of
finite injective dimension and those of finite 
$\catfcc$-projective dimension
by~\cite[Cors.\ 6.1 and 6.2]{holm:fear}.

Arguing as in~\cite[(3.2.9)]{christensen:gd}, we 
see that
$M\in\catac$ if and only if $M^*\in\catbc$,
and
$M\in\catbc$ if and only if $M^*\in\catac$.
Similarly, we have $M\in\catbc$ if and only if $\hom_R(C,M)\in\catac$
by~\cite[(2.8.a)]{takahashi:hasm}.
From~\cite[Thm.\ 6.1]{holm:fear} we know that every module in $\catbc$
has a $\catpcc$-proper
$\catpcc$-resolution.  
\end{fact}

The next definitions are due to Holm and J\o rgensen~\cite{holm:smarghd} 
in this generality.  

\begin{defn} \label{defn0301}
Let $C$ be a semidualizing $R$-module.
A \emph{complete $\caticc\catii$-resolution} is a complex $Y$ of $R$-modules 
satisfying the following:
\begin{enumerate}[\quad(1)]
\item $Y$ is exact and $\Hom_R(\caticc,-)$-exact, and
\item $Y_i$ is $C$-injective when $i\geq 0$ and $Y_i$ is  injective when  $i< 0$.
\end{enumerate}
An $R$-module $H$ is \emph{$\text{G}_C$-injective} if there
exists a complete $\caticc\catii$-resolution $Y$ such that $H\cong\coker(\partial^Y_1)$,
in which case $Y$ is a \emph{complete $\caticc\catii$-resolution of $H$}.  
We set
$$\catgic=\text{the subcategory of $\text{G}_C$-injective $R$-modules}.$$
In the special case $C=R$, we write
$\catgi$ in place of $\catgir$.

A \emph{complete $\catff\catfcc$-resolution} is a complex $Z$ of $R$-modules 
satisfying the following.
\begin{enumerate}[\quad(1)]
\item $Z$ is exact and $-\otimes_R\caticc$-exact.
\item $Z_i$ is flat if  $i\geq 0$ and $Z_i$ is  $C$-flat if $i< 0$.
\end{enumerate}
An $R$-module $M$ is \emph{$\text{G}_C$-flat} if there
exists a complete $\catff\catfcc$-resolution $Z$ such that $M\cong\coker(\partial^Z_1)$,
in which case $Z$ is a \emph{complete $\catff\catfcc$-resolution of $M$}.  We set
$$\catgfc=\text{the subcategory of $\text{G}_C$-flat $R$-modules}.$$
In the special case $C=R$, we set
$\catgf=\catgfr$,
and $\gfd=\catgff\text{-}\pd$.

A \emph{complete $\catpp\catpcc$-resolution} is a complex $X$ of $R$-modules 
satisfying the following.
\begin{enumerate}[\quad(1)]
\item $X$ is exact and $\Hom_R(-,\catpcc)$-exact.
\item $X_i$ is projective if  $i\geq 0$ and $X_i$ is  $C$-projective if $i< 0$.
\end{enumerate}
An $R$-module $M$ is \emph{$\text{G}_C$-projective} if there
exists a complete $\catpp\catpcc$-resolution $X$ such that $M\cong\coker(\partial^X_1)$,
in which case $X$ is a \emph{complete $\catpp\catpcc$-resolution of $M$}.  We set
$$\catgpc=\text{the subcategory of $\text{G}_C$-projective $R$-modules}.$$
\end{defn}

\begin{fact} \label{fact0201}
Let $C$ be a semidualizing $R$-module.
Flat $R$-modules and $C$-flat $R$-modules are $\text{G}_C$-flat by~\cite[(2.8.c)]{holm:smarghd}.
It is straightforward to show that an $R$-module $M$ is $\text{G}_C$-flat if and only
the following conditions hold:
\begin{enumerate}[(1)]
\item \label{fact0201a}
$M$ admits an augmented $\catfcc$-coresolution
that is $-\otimes_R\caticc$-exact, and
\item \label{fact0201b}
$\tor_{\geq 1}^R(M,\caticc)=0$.
\end{enumerate}
Let $R\ltimes C$ denote the trivial extension of $R$ by $C$, defined to be the $R$-module
$R\ltimes_R C=R\oplus C$ with ring structure given by $(r,c)(r',c')=(rr',rc'+r'c)$.
Each $R$-module $M$ is naturally an $R\ltimes C$-module via the
natural surjection $R\ltimes C\to R$. 
Within this protocol we have
$M\in\catgic$ if and only if $M\in\cat{GI}(R\ltimes C)$ and
$M\in\catgfc$ if and only if $M\in\cat{GF}(R\ltimes C)$ by~\cite[(2.13) and (2.15)]{holm:smarghd}.
Also~\cite[(2.16)]{holm:smarghd} implies
$\gfcpd_R(M)=\gfd_{R\ltimes C}(M)$.
\end{fact}

The next definition, from~\cite{sather:sgc}, is modeled on the construction of $\catgi$.

\begin{defn} \label{defn0502}
Let $\catx$ be a subcategory
of $\catm$. A \emph{complete $\catx$-resolution} is an
exact complex $X$ in $\catx$ that is $\hom_R(\catx,-)$-exact and
$\hom_R(-,\catx)$-exact.\footnote{In the literature, these complexes
are sometimes called ``totally acyclic''.}   
Such a complex is a \emph{complete $\catx$-resolution} 
of $\coker(\partial^X_1)$.
We set 
$$\catg(\catx)=\text{the subcategory of $R$-modules
with a complete $\catx$-resolution}.$$
Set $\catg^0(\catx)=\catx$, $\catg^1(\catx)=\catg(\catx)$ and
$\catg^{n+1}(\catx)=\catg(\catg^n(\catx))$
for $n\geq 1$.
\end{defn}

\begin{fact}  \label{fact0501}
Let $\catx$ be a subcategory
of $\catm$. Using a resolution of the form $0\to X\to 0$,
one sees  that
$\catx\subseteq\catg(\catx)$ and so
$\catg^n(\catx)\subseteq\catg^{n+1}(\catx)$ for each $n\geq 0$.
If $C$ is a semidualizing $R$-module, then
$\catg^n(\catic)=\catgic\cap\catac$
for each $n\geq 1$;
see~\cite[(5.5)]{sather:sgc}.
\end{fact}

The final definition of this section is for use in the proof of Theorem~\ref{thmb}.

\begin{defn} \label{defn0501}
Let $C$ be a semidualizing $R$-module, and let $\catx$ be a  subcategory
of $\catm$. A \emph{$\catpcc\catfccott$-complete $\catx$-resolution} is an
exact complex $X$ in $\catx$ that is $\hom_R(\catpcc,-)$-exact and
$\hom_R(-,\catfccott)$-exact.  
Such a complex is a \emph{$\catpcc\catfccott$-complete $\catx$-resolution} 
of $\coker(\partial^X_1)$.
We set 
$$\cathc(\catx)=\text{the subcategory of $R$-modules
with a $\catpcc\catfccott$-complete $\catx$-resolution}.$$
Set $\cathc^0(\catx)=\catx$, $\cathc^1(\catx)=\cathc(\catx)$ and
$\cathc^{n+1}(\catx)=\cathc(\cathc^n(\catx))$
for each $n\geq 1$.
\end{defn}

\begin{disc}  \label{disc0501}
Let $C$ be a semidualizing $R$-module, and let $\catx$ be a  subcategory
of $\catm$.  Let $X$ be an 
exact complex  in $\catx$ that is $\hom_R(C,-)$-exact and
$\hom_R(-,\catfccott)$-exact.  
Hom-tensor adjointness implies that
$X$ is
$\hom_R(\catpcc,-)$-exact and hence a $\catpcc\catfccott$-complete $\catx$-resolution,
as is the complex $\shift^iX$ for each $i\in\bbz$. It follows that
$\coker(\partial^X_i)\in\cathc(\catx)$
for each $i$.

Using a resolution of the form $0\to X\to 0$,
one sees  that $\catx\subseteq\cathc(\catx)$ and so
$\cathc^n(\catx)\subseteq\cathc^{n+1}(\catx)$ for each $n\geq 0$.
Furthermore, if  $\catfc\subseteq\catx$, then
$\catg(\catx)\subseteq\cathc(\catx)$ and so
$\catg^n(\catx)\subseteq\cathc^n(\catx)$ for each $n\geq 1$.
\end{disc}

\section{Modules of finite $\catfccott$-projective dimension} \label{sec04}

This section contains the fundamental properties of the 
modules of finite $\catfccott$-projective dimension. The first two results allow us to 
deduce information for these modules
from the modules of finite $\catic$-injective dimension.

\begin{lem}\label{pdual}
Let $M$ be an $R$-module,
and let $C$ be a semidualizing $R$-module.
\begin{enumerate}[\quad\rm(a)]
\item \label{pduala}
The Pontryagin dual
$M^*$ is $C$-flat  if and only if  $M$ is $C$-injective.
\item \label{pdualb}
The Pontryagin dual
$M^*$ is $C$-injective  if and only if  $M$ is $C$-flat.
\item \label{pdualc}
If $\tor^R_{\geq 1}(C,M)=0$, then $M^*$ is $C$-cotorsion.
\item \label{pduald}
If $M$ is $C$-injective, then $M^*$ is $C$-flat and $C$-cotorsion.
\end{enumerate}
\end{lem}

\begin{proof}
\eqref{pduala}
Assume that $M$ is $C$-injective, so there exists an injective $R$-module $I$ such
that $M\cong\hom_R(C,I)$. This yields the first isomorphism in the following sequence
while the second is from Hom-evaluation~\cite[(0.3.b)]{christensen:apac}:
$$M^*
\cong\Hom_{\bbz}(\hom_R(C,I),\bbq/\bbz)
\cong C\otimes_R\Hom_{\bbz}(I,\bbq/\bbz).
$$
Since $I$ is injective, Lemma~\ref{lem0202}\eqref{lem0202b}
implies that $\Hom_{\bbz}(I,\bbq/\bbz)$ is flat.
Hence, the displayed isomorphisms imply that $M^*$
is $C$-flat.

Conversely, assume that $M^*$ is $C$-flat, so there exists
a flat $R$-module $F$ such that $M^*\cong F\otimes_R C$.
As $F$ is flat it is in $\catac$, and this yields the first isomorphism in the next sequence,
while the third isomorphism is Hom-tensor adjointness
$$F\cong\Hom_R(C,F\otimes_R C)
\cong\Hom_R(C,\Hom_{\bbz}(M,\bbq/\bbz))
\cong\Hom_{\bbz}(C\otimes_RM,\bbq/\bbz).
$$
This module is flat, and so Lemma~\ref{lem0202}\eqref{lem0202a}
implies that $C\otimes_RM$ is injective.
From~\cite[Thm.\ 1]{holm:fear} we conclude that
$M$ is $C$-injective.

\eqref{pdualb}
This is proved similarly.

\eqref{pdualc}
Let $P$ be a projective resolution of $M$.
Our Tor-vanishing hypothesis implies that there is a quasiisomorphism
$C\otimes_R P\simeq C\otimes_RM$.
For each flat $R$-module $F$,  this yields a quasiisomorphism
$$F\otimes_RC\otimes_R P\simeq F\otimes_RC\otimes_RM.$$
Because $\qmodz$ is injective over $\bbz$, this provides the third 
quasiisomorphism in the next
sequence, while
the second quasiisomorphism is Hom-tensor adjointness
\begin{align}
\Hom_R(F\otimes_RC,P^*)
&\simeq\Hom_R(F\otimes_RC,\Hom_{\bbz}(P,\qmodz))\notag\\
&\simeq\Hom_{\bbz}(F\otimes_RC\otimes_RP,\qmodz)\label{pdulac1}\tag{$\ast$}\\
&\simeq\hom_{\bbz}(F\otimes_RC\otimes_RM,\qmodz).\notag
\end{align}
Since $\qmodz$ is injective over $\bbz$, there are quasiisomorphisms
$$M^* \simeq\pdual{M}\simeq\Hom_{\bbz}(P,\qmodz)
\simeq P^*.$$
By Lemma~\ref{lem0202}\eqref{lem0202a}, it follows
that $P^*$ is an injective resolution
of $M^*$ over $R$.
In particular, taking cohomology in the displayed sequence~\eqref{pdulac1} yields
isomorphisms
\begin{align*}
\ext^{i}_R(F\otimes_RC,M^*)
&\cong\HH_{-i}(\Hom_R(F\otimes_RC,P^*))\\
&\cong\HH_{-i}(\hom_{\bbz}(F\otimes_RC\otimes_RM,\qmodz)).
\end{align*}
This is 0 when $i\neq 0$
because 
$\hom_{\bbz}(F\otimes_RC\otimes_RM,\qmodz)$
is a module.
Hence, the desired conclusion.

\eqref{pduald}
Since $M$ is $C$-injective, it is in $\catac$ by Fact~\ref{projac},
and so
$\tor_{\geq 1}^R(C,M)=0$.
Hence $M$ is $C$-cotorsion by part~\eqref{pdualc},
and it is $C$-flat by part~\eqref{pduala}.
\end{proof}

\begin{lem}\label{lem0701}
Let $M$ be an $R$-module,
and let $C$ be a semidualizing $R$-module.
\begin{enumerate}[\quad\rm(a)]
\item \label{lem0701a}
There is an equality 
$\icid_R(M^*)=\fcpd_R(M)$.
\item \label{lem0701b}
There is an equality 
$\fcpd_R(M^*)=\icid_R(M)$.
\end{enumerate}
\end{lem}

\begin{proof}
We prove part~\eqref{lem0701a};
the proof of part~\eqref{lem0701b} is similar.

For the inequality $\icid_R(M^*)\leq\fcpd_R(M)$,
assume  that $\fcpd_R(M)<\infty$.
Let $X$ be a $\catfc$-resolution of $M$ such that
$X_i=0$ for all $i>\fcpd_R(M)$.
It follows from Lemma~\ref{pdual}\eqref{pdualb}
that the complex $X^*$ is an $\caticc$-coresolution
of $M^*$ such that $X_i^*=0$ for all $i>\fcpd_R(M)$.
The desired inequality now follows.

For the reverse inequality, 
assume  that $j=\icid_R(M^*)<\infty$.
Fact~\ref{projac} implies that $M^*$ is in $\catac$,
and hence also implies that $M\in\catbc$. This 
condition implies that $M$ has a $\catpcc$-resolution $Z$
by Fact~\ref{projac}.
In particular, this is an $\catfcc$-resolution of $M$,
and so Lemma~\ref{pdual}\eqref{pdualb} implies that
$Z^*$ is an $\caticc$-coresolution of $M^*$. 
From~\cite[(3.3.b)]{takahashi:hasm}
we know that $\ker((\partial^Z_{j+1})^*)\cong\coker(\partial^Z_{j+1})^*$
is in $\catic$. Lemma~\ref{pdual}\eqref{pdualb} implies
$\coker(\partial^Z_{j+1})\in\catfc$.
It follows that the truncated complex
$$Z': 0\to\coker(\partial^Z_{j+1})\to Z_{j-1}\to \cdots\to Z_0\to 0$$
is an $\catfcc$-resolution of $M$ such that
$Z_i'=0$ for all $i>j$.
The desired inequality now follows, and hence the equality.
\end{proof}

The next three lemmata document properties of $\catfccot$
for use in the sequel.  The first of these contains the 
characterization of $C$-flat $C$-cotorsion modules
mentioned in Remark~\ref{defn0401}.

\begin{lem} \label{lem0201}
Let $C$ and $M$ be  $R$-modules with $C$ semidualizing.
The following conditions are equivalent:
\begin{enumerate}[\quad\rm(i)]
\item \label{lem0201a}
$M\in\catfccot$;
\item \label{lem0201b}
$M\in\catfc$ and $\catfc\perp M$;
\item \label{lem0201c}
$M\in\catbc$ and $\hom_R(C,M)\in\catfcot$;
\item \label{lem0201d}
$\hom_R(C,M)\in\catfcot$.
\end{enumerate}
In particular, we have $\catfc\perp\catfccot$.
\end{lem}

\begin{proof}
\eqref{lem0201a}$\iff$\eqref{lem0201b}.
It suffices to show, for each flat $R$-module $F$,
that $\catf\perp F$ if and only if $\catfc\perp F\otimes_RC$.
Let $F'$ be a flat $R$-module.
It suffices to show that 
\begin{equation*} 
\ext^{i}_R(F'\otimes_RC,F\otimes_RC)
\cong \ext^i_R(F',F)
\end{equation*}
for each $i$.
From~\cite[(1.11.a)]{white:gpdrsm}
we have the first
isomorphism in the next sequence
$$\ext^i_R(C,F\otimes_RC)
\cong\ext^i_R(C,C)\otimes_R F
\cong\begin{cases}
R\otimes_RF\cong F & \text{if $i\neq 0$} \\
0\otimes_R F\cong 0 & \text{if $i= 0$}
\end{cases}
$$
and the second isomorphism is from the fact that $C$ is semidualizing.
Let $P$ be a projective resolution of $C$.
The previous display provides a quasiisomorphism
$$\Hom_R(P,F\otimes_R C)\simeq F.$$
Let $P'$ be a projective resolution of $F'$.  
Hom-tensor adjointness
yields the first quasiisomorphism in the next sequence
$$\Hom_R(P'\otimes_RP,F\otimes_RC)
\simeq\Hom_R(P',\Hom_R(P,F\otimes_RC))
\simeq\Hom_R(P',F)$$
and the second quasiisomorphism is from the previous display,
because $P'$ is a bounded below complex of projective
$R$-modules.
Since $F'$ is flat, we conclude that $P'\otimes_RP$ is a projective resolution
of $F'\otimes_R C$.
It follows that we have
\begin{align*}
\ext^{i}_R(F'\otimes_RC,F\otimes_RC)
&\cong\HH_{-i}(\Hom_R(P'\otimes_RP,F\otimes_RC)) \\
&\cong\HH_{-i}(\Hom_R(P',F))\\
&\cong \ext^i_R(F',F)
\end{align*}
as desired.

\eqref{lem0201a}$\implies$\eqref{lem0201c}.
Assume that $M\in\catfccot$, that is, that $M\cong C\otimes_R F$ for some
$F\in\catfcot\subseteq\catac$.
Then
$$\hom_R(C,M)\cong\hom_R(C, C\otimes_R F)\cong F\in\catfccot$$
and $M\in\catfccot\subseteq\catfc\subseteq\catbc$.

\eqref{lem0201c}$\implies $\eqref{lem0201a}.
If $M\in\catbc$ and $\hom_R(C,M)\in\catfcot$,
then there is an isomorphism $M\cong C\otimes_R\hom_R(C,M)\in\catfccot$.

\eqref{lem0201c}$\iff$\eqref{lem0201d}. This is from Fact~\ref{projac}
because $\catfcot\subseteq\catac$.

The conclusion $\catfc\perp\catfccot$ follows from the implication
\eqref{lem0201a}$\implies$\eqref{lem0201b}.
\end{proof}

\begin{lem} \label{lem0207}
If $C$ is a semidualzing $R$-module,
then the category $\catfccot$ is closed under products, extensions and summands.
\end{lem}

\begin{proof}
Consider a set $\{F_{\lambda}\}_{\lambda\in\Lambda}$ of
modules in $\catfcot$.
From~\cite[(3.2.24)]{enochs:rha} we have
$\prod_{\lambda}F_{\lambda}\in \catfcot$
and so $C\otimes_R(\prod_{\lambda}F_{\lambda})\in \catfccot$.
Hence, we have
$$\textstyle
\prod_{\lambda}(C\otimes_RF_{\lambda})
\cong C\otimes_R(\prod_{\lambda}F_{\lambda})\in \catfccot$$
where the isomorphism comes from the
fact that $C$ is finitely presented. Thus $\catfccot$ is closed under products.

By Lemma~\ref{perp03}\eqref{perp03item2}, the category of
$C$-cotorsion $R$-modules is closed under extensions, and 
it is closed under summands by the additivity
of Ext.
The category $\catfc$ is closed under extensions and summands
by~\cite[Props.\ 5.1(a) and 5.2(a)]{holm:fear}.
The result now follows from Lemma~\ref{lem0201}.
\end{proof}

Note that hypotheses of the next lemma are satisfied when $M\in \catfc^{\perp}\cap\catbc$.

\begin{lem} \label{lem0206}
Let $C$ be a semidualizing $R$-module, and 
let $M$ be a $C$-cotorsion $R$-module such that the natural evaluation map
$C\otimes_R\Hom_R(C,M)\to M$ is bijective.
\begin{enumerate}[\quad\rm(a)]
\item \label{lem0206a}
The module $M$ has an $\catfccott$-cover,
and every $C$-flat cover of $M$ is an $\catfccott$-cover of $M$ with $C$-cotorsion
kernel.
\item \label{lem0206b}
Each $\catfccott$-precover of $M$ is surjective.
\item \label{lem0206c}
Assume further that $\tor_{\geq 1}^R(C,\Hom_R(C,M))=0$.
Then $M$ has an $\catfcc$-proper $\catfccott$-resolution 
such that $\ker(\partial^X_{i-1})$ is $C$-cotorsion for each $i$.
\end{enumerate}
\end{lem}

\begin{proof}
\eqref{lem0206a}
The module $M$ has a $C$-flat cover $\vf\colon F\otimes_RC\to M$
by~\cite[Prop.\ 5.3.a]{holm:fear}, and $\ker(\vf)$ is $C$-cotorsion
by~\cite[(2.1.1)]{xu:fcm}. Furthermore, the bijectivity of the evaluation
map
$C\otimes_R\Hom_R(C,M)\to M$ implies that there is a 
projective $R$-module $P$ and a surjective
map $\vf'\colon P\otimes_RC\onto M$ by~\cite[(2.2.a)]{takahashi:hasm}.
The fact that $\vf$ is a precover provides a map $f\colon P\otimes_R C\to F\otimes_R C$
such that $\vf'=\vf f$. Hence, the surjectivity of $\vf'$ implies that $\vf$
is surjective.  
It follows from Lemma~\ref{perp03}\eqref{perp03item1} that
$F\otimes_R C$ is $C$-cotorsion, and so
$F\otimes_R C\in\catfccot$ by Lemma~\ref{lem0201}.  
Since $\vf$ is a $C$-flat cover and
$\catfccot\subseteq\catfc$, we conclude that $\vf$
is an $\catfccott$-cover.

\eqref{lem0206b}
This follows as in part~\eqref{lem0206a} because $M$ has a surjective
$\catfccott$-cover.

\eqref{lem0206c}
Using parts~\eqref{lem0206a} and~\eqref{lem0206b},
the argument of~\cite[Thm.\ 2]{holm:fear} shows how to construct
a resolution with the desired properties.
\end{proof}

The final three results of this section contain our main conclusions 
for $\finrescatfccot$.
The first of these extends Lemma~\ref{lem0201}.

\begin{prop} \label{lem0401}
Let $C$ and $M$ be  $R$-modules with $C$ semidualizing,
and let $n\geq 0$.
The following conditions are equivalent:
\begin{enumerate}[\qquad\rm(i)]
\item \label{lem0401a}
$\fccotpd_R(M)\leq n$;
\item \label{lem0401b}
$M\in\catbc$ and $\fcotpd_R(\hom_R(C,M))\leq n$;
\item \label{lem0401e}
$\fcotpd_R(\hom_R(C,M))\leq n$;
\item \label{lem0401c}
$M\cong C\otimes_R K$ for some $R$-module $K$ such that
$\fcotpd_R(K)\leq n$;
\item \label{lem0401d}
$\fcpd_R(M)\leq n$ and $\catfc\perp M$.
\end{enumerate}
\end{prop}

\begin{proof}
\eqref{lem0401a}$\implies$\eqref{lem0401b}
Since $\fccotpd_R(M)\leq n<\infty$, we have $M\in\catbc$
by Fact~\ref{projac}.
Let $X$ be an $\catfccott$-resolution of $M$
such that $X_i=0$ when $i>n$. for each $i$, let $F_i\in\catfcot$
such that $X_i\cong F_i\otimes_R C$.  
Since each $F_i$ is in $\catac$, we have
$$\Hom_R(C,X)_i\cong\Hom_R(C,X_i) \cong\Hom_R(C,F_i\otimes_R C)
\cong F_i.$$
A standard argument using the conditions $M,X_i\in\catbc$ shows that
$\Hom_R(C,X)$ is an $\catfcott$-resolution of $\hom_R(C,M)$
such that $\Hom_R(C,X)_i=0$ when $i>n$.
The inequality $\fcotpd_R(\hom_R(C,M))\leq n$
then follows.

\eqref{lem0401b}$\implies$\eqref{lem0401c}
The condition $M\in\catbc$ implies $M\cong C\otimes_R\hom_R(C,M)$,
and so $K=\hom_R(C,M)$ satisfies the desired conclusions.

\eqref{lem0401c}$\implies$\eqref{lem0401d}
Let $F$ be an $\catfcott$-resolution of $K$
such that $F_i=0$ when $i>n$. 
Using the condition $K,F_i\in \catac$,
a standard argument shows
that $C\otimes_R F$ 
is an $\catfccott$-resolution of $C\otimes_R K\cong M$.
Hence,  this resolution yields $\fcpd_R(M)\leq\fccotpd_R(M)\leq n$.
By Lemma~\ref{lem0201}, we have $\catfc\perp\catfccot$,
and so Lemma~\ref{gencat01} implies $\catfc\perp\finrescatfccot$;
in particular $\catfc\perp M$.

\eqref{lem0401d}$\implies$\eqref{lem0401a}
The assumption $\fcpd_R(M)\leq n$ implies $M\in\catbc$ by Fact~\ref{projac}, and so
$\ext^{\geq 1}_R(C,M)=0$. 
Lemma~\ref{lem0206}\eqref{lem0206c}
implies that $M$ has an $\catfcc$-proper $\catfccott$-resolution $X$
such that 
$K_i=\ker(\partial^X_{i-1})$ 
is $C$-cotorsion for each $i$. 
In particular, the truncated complex
$$X'=\qquad 0\to K_{n}\to X_{n-1}\to\cdots\to X_0\to M\to 0$$
is exact and $\hom_R(C,-)$-exact.
Since $\fcpd_R(M)\leq n$, 
the proof of the implication
\eqref{lem0401a}$\implies$\eqref{lem0401b}
shows that $\fd_R(\Hom_R(C,M))\leq n$.  
Since each $R$-module $\hom_R(C,X_i)$ is flat by Lemma~\ref{lem0201},
the exact complex $\hom_R(C,X')$ is a truncation of an augmented 
flat resolution of $\Hom_R(C,M)$. It follows that
$\hom_R(C,K_n)$ is flat, and so
$K_n\in\catfc$ by~\cite[Thm.\ 1]{holm:fear}.
Hence 
$X'$ is an augmented $\catfccott$-resolution of $M$,
and so $\fccotpd_R(M)\leq n$.

\eqref{lem0401b}$\iff$\eqref{lem0401e} follows from Fact~\ref{projac}
because $\finrescatfcot\subseteq\catac$.
\end{proof}

\begin{lem} \label{lem0211}
Let $C$ be a semidualizing $R$-module.
If $\fccotpd_R(M)<\infty$, then any bounded $\catfccott$-resolution
$X$ of $M$ is $\catfcc$-proper.
\end{lem}

\begin{proof}
Observe that $\catfc\perp X_i$ for all $i$ 
and $\catfc\perp M$
by Proposition~\ref{lem0401}.
So, the complex $X^+$ is exact and such that
$(X^+)_i=0$ for $i\gg 0$ and $\catfc\perp (X^+)_i$.
Hence, Lemma~\ref{lem0216}\eqref{lem0216b}
implies that $X^+$ is
$\hom_R(\catfcc,-)$-exact.
\end{proof}

\begin{prop} \label{lem0210}
Let $C$ be a semidualizing $R$-module.
The category $\finrescatfccot$ is
closed under extensions, cokernels of monomorphisms and  summands.
\end{prop}

\begin{proof}
Consider an exact sequence
$$0\to M_1\to M_2\to M_3\to 0$$
such that $\fccotpd_R(M_1)$ and $\fccotpd_R(M_3)$ are finite.
To show that $\finrescatfccot$ is
closed under extensions 
we need to show that $\fccotpd_R(M_2)$ is finite.

The condition $\fccotpd_R(M_1)<\infty$ implies
$\icid(M_1^*)=\fcpd_R(M_1)<\infty$ by Lemma~\ref{lem0701}\eqref{lem0701a}
and Proposition~\ref{lem0401};
and similarly $\icid(M_3^*)<\infty$.
From~\cite[(3.4)]{takahashi:hasm} we know that the category of
$R$-modules of finite $\caticc$-injective dimension is closed under extensions.
Using the dual exact sequence
$$0\to M_3^*\to M_2^*\to M_1^*\to 0$$
we conclude
that $\icid(M_2^*)$ is  finite.
Thus, Lemma~\ref{lem0701}\eqref{lem0701a}
implies that
that  $\fcpd_R(M_2)$ is finite.

Since $\fccotpd_R(M_1)<\infty$, Proposition~\ref{lem0401}
implies $\catfc\perp M_1$; and similarly $\catfc\perp M_3$.
Thus, we have $\catfc\perp M_2$
by Lemma~\ref{perp03}\eqref{perp03item2}.
Combining this with the previous paragraph,
Proposition~\ref{lem0401} implies that $\fccotpd_R(M_2)<\infty$.

The proof of the fact that $\finrescatfccot$ is
closed under cokernels of monomorphisms is similar.
The fact  that $\finrescatfccot$ is
closed under  summands is even easier to prove using
the natural isomorphism $(M_1\oplus M_2)^*\cong M_1^*\oplus M_2^*$.
\end{proof}

\section{Weak AB-Context} \label{sec03}

Let $C$ be a semidualizing $R$-module.
The point of this section is to show that the triple
$(\catgfc,\finrescatfccot,\catfccot)$ is a weak AB-context,
and to document the immediate consequences;
see Theorem~\ref{thma} and Corollary~\ref{cor0301}.
We begin the section with two results  modeled on~\cite[(3.22) and (3.6)]{holm:ghd}.

\begin{lem}\label{extgfcvan}
If $C$ is a semidualizing $R$-module, then
$\catgfc\perp\finrescatfccot$.
\end{lem}

\begin{proof}
By Lemma~\ref{gencat01} it suffices to show
$\catgfc\perp\catfccot$.
Fix modules $M\in\catgfc$ and 
$N\in\catfccot$. 
By Lemma~\ref{pdual}, we know that
the Pontryagin dual $N^*$ is $C$-injective.
Hence, for $i\geq 1$, the vanishing in the next sequence is from
Fact~\ref{fact0201}
$$
\ext^i_R(M,N^{**})
\cong\ext^i_R(M,\pdual{N^*})
\cong\pdual{\tor_R^i(M,N^*)}
=0.
$$
The second isomorphism is a form of Hom-tensor adjointness using the
fact that $\qmodz$ is injective over $\bbz$.
To finish the proof, it suffices to show that $N$ is a summand of $N^{**}$;
then the last sequence shows $\ext^{\geq 1}_R(M,N)=0$.
Write $N\cong C\otimes_R F$ for some flat cotorsion $R$-module $F$, and use
Hom-tensor adjointness to conclude
$$N^*\cong
\pdual{C\otimes_R F}
\cong\hom_R(C,\pdual{F}).$$
Lemma~\ref{lem0202}\eqref{lem0202b} implies that $\pdual{F}$ is injective,
so the proof of Lemma~\ref{pdual}\eqref{pduala} 
explains the second isomorphism in the next sequence
$$N^{**}\cong\hom_R(C,\pdual{F})^*
\cong C\otimes_R\pdual{\pdual{F}}
\cong C\otimes_R F^{**}.
$$
The proof of~\cite[(3.22)]{holm:ghd} shows that 
$F$ is a summand of $F^{**}$,
and it follows that $N\cong C\otimes_R F$ is a summand of 
$C\otimes_R F^{**}\cong N^{**}$, as desired.
\end{proof}

\begin{lem}\label{pdual2}
Let $C$ be a semidualizing $R$-module. If $M$ is an $R$-module, then
$M$ is in $\catgfc$ if and only if 
its Pontryagin dual $M^*$ is
in $\catgic$.
\end{lem}

\begin{proof} 
Consider the trivial extension $R\ltimes C$ from Fact~\ref{fact0201}.
By~\cite[(3.6)]{holm:ghd} we know that
$M$ is in $\catgff(R\ltimes C)$ if and only if 
$M^*$ is
in $\catgii(R\ltimes C)$.
Also $M$ is in $\catgff(R\ltimes C)$ if and only if 
$M$ is in $\catgfc$, and $M^*$ is in $\catgii(R\ltimes C)$ if and only if 
$M^*$ is in $\catgic$ by Fact~\ref{fact0201}. Hence, the equivalence.
\end{proof}

The following result establishes Theorem~\ref{thma}\eqref{thma1}.

\begin{prop}\label{gfcprojres}
Let $C$ be a semidualizing $R$-module.
The category 
$\catgfc$ is closed under kernels of epimorphisms,
extensions and  summands.
\end{prop}

\begin{proof}
The result dual to~\cite[(3.8)]{white:gpdrsm}
says that $\catgic$ is closed under cokernels of monomorphisms,
extensions and  summands.
To see that $\catgfc$ is closed under  summands,
let $M\in\catgfc$ and assume that $N$ is a direct summand of $M$.
It follows that the Pontryagin dual $N^*$ is a direct summand of $M^*$.
Lemma~\ref{pdual2} implies that $M^*$ is in $\catgic$ which is closed
under  summands.  We conclude that $N^*\in\catgic$,
and so $N\in\catgfc$.  Hence 
$\catgfc$ is closed under  summands,
and the other properties are verified similarly.
\end{proof}

The next four results put the finishing touches on Theorem~\ref{thma}.

\begin{lem} \label{lem0215}
Let $C$ be a semidualizing $R$-module.
If $X$ is a complete $\catff\catfcc$-resolution,
then $\coker(\partial^X_n)\in\catgfc$ for each $n\in\bbz$.
\end{lem}

\begin{proof}
Write $M_n=\coker(\partial^X_n)$, and note that $M_1\in\catgfc$ by definition.
Fact~\ref{fact0201} implies that $X_n\in\catgfc$ for 
each $n\in\bbz$.  Since $M_1$ is in $\catgfc$, an induction argument using
Proposition~\ref{gfcprojres} shows $M_n\in\catgfc$ for each $n\geq 1$.

Now assume $n\leq 0$.  
Lemma~\ref{lem0216}\eqref{lem0216c},
implies $\tor_{\geq 1}^R(M_n,\caticc)=0$.
By construction, the following sequence is exact
and $-\otimes_R\caticc$-exact
$$0\to M_n\to X_{n-2}\to X_{n-3}\cdots$$
with each $X_{n-i}\in\catgfc$,
and so $M_n\in\catgfc$ by
Fact~\ref{fact0201}. 
\end{proof}

\begin{lem}\label{lem1001}
Let $C$ be a semidualizing $R$-module.
If $M\in\catfc$, then there is an exact sequence
$0\to M\to M_1\to M_2\to 0$
with $M_1\in\catfccot$ and $M_2\in\catfc$.
\end{lem}

\begin{proof}
Since $M$ is $C$-flat, we know from~\cite[Thm.\ 1]{holm:fear}
that $\hom_R(C,M)$ is flat.
By~\cite[(3.1.6)]{xu:fcm} there is a cotorsion flat module $F$ containing 
$\hom_R(C,M)$ such
that the quotient $F/\hom_R(C,M)$
is flat.
Consider the exact sequence
$$0\to \hom_R(C,M)\to F\to F/\hom_R(C,M)\to 0.$$
Since $F/\hom_R(C,M)$ is flat, an application of $C\otimes_R-$
yields an exact sequence
$$0\to C\otimes_R\hom_R(C,M)\to C\otimes_RF\to C\otimes_R(F/\hom_R(C,M))\to 0.$$
Because $M$ is $C$-flat, it is in $\catbc$ and so
$C\otimes_R\hom_R(C,M)\cong M$. With
$M_1=C\otimes_RF$ and $M_2=C\otimes_R(F/\hom_R(C,M))$
this yeilds the desired sequence.
\end{proof}

\begin{lem}\label{fccotpreenv}
Let $C$ be a semidualizing $R$-module.
Each module $M\in\catgfc$ admits an injective
$\catfccott$-preenvelope $\alpha\colon M\to Y$ such that $\coker(\alpha)\in\catgfc$.
\end{lem}

\begin{proof}
Let $M\in\catgfc$ with complete $\catff\catfcc$-resolution  $X$.
By definition, this says that $M$ is a submodule of the
$C$-flat $R$-module $X_{-1}$, and Lemma~\ref{lem0215}
implies that $X_{-1}/M\in\catgfc$.
Since $X_{-1}$ is $C$-flat, 
Lemma~\ref{lem1001} yields an exact sequence
$$0\to X_{-1}\to Z\to Z/X_{-1}\to 0$$
with $Z\in\catfccot$ and $Z/X_{-1}\in\catfc$.
It follows that $Z/X_{-1}$ is in $\catgfc$.
Since $X_{-1}/M$ is also  in $\catgfc$, and $\catgfc$ is closed under extensions
by Proposition~\ref{gfcprojres},
the following exact sequence shows that
$Z/M$ is also in $\catgfc$ 
$$0\to X_{-1}/M\to Z/M\to Z/X_{-1}\to 0.$$
In particular, Lemma~\ref{extgfcvan} implies
$Z/M\perp\catfccot$,
and it follows that the next sequence is
$\hom_R(-,\catfccott)$-exact by Lemma~\ref{perp03}\eqref{perp03item1}.
$$0\to M\to C\otimes_R F\to Z/M\to 0$$
The conditions $Z\in \catfccot$ 
and $Z/M\in\catgfc$
then
implies that the inclusion $M\to Z$ is 
an $\catfccott$-preenvelope whose cokernel is in $\catgfc$.
\end{proof}

\begin{prop}  \label{lem0303}
Let $C$ be a semidualizing $R$-module.
The category $\catfccot$ is an injective cogenerator for the category
$\catgfc$. In particular, every module in $\catgfc$ admits 
a $\catfccott$-proper $\catfccott$-coresolution, and so $\catgfc\subseteq\propcorescatfccot$. 
\end{prop}

\begin{proof}
Lemmas~\ref{extgfcvan} and~\ref{fccotpreenv} 
imply that $\catfccot$ is an injective cogenerator for 
$\catgfc$.
The remaining conclusions follow immediately.
\end{proof}

\begin{lem}  \label{lem0304}
If $C$ is a semidualizing $R$-module,
then there is an equality
$\catfccot=\catgfc\cap\finrescatfccot$.
\end{lem}

\begin{proof}
The containment $\catfccot\subseteq\catgfc\cap\finrescatfccot$ is
straightforward; see Definition~\ref{notation03}
and Fact~\ref{fact0201}. For the reverse containment, let
$M\in\catgfc\cap\finrescatfccot$. Truncate a bounded
$\catfccott$-resolution to obtain an exact sequence
$$0\to K \to F\otimes_R C \to M \to 0$$ 
with $F\in\catfcot$ 
and such that $\fccotpd_R(K)<\infty$.
We have $\ext^1_R(M,K)=0$ by Lemma~\ref{extgfcvan}, so this
sequence splits.  
Hence $M$ is a summand of $F\otimes_R C\in\catfccot$.
Lemma~\ref{lem0207} implies that $\catfccot$ is closed under summands,
so $M\in \catfccot$.
\end{proof}

\begin{para}\label{thm0301}
\textit{Proof of Theorem~\ref{thma}.}
Part~\eqref{thma1} is in Proposition~\ref{gfcprojres}.
Since $\catfccot\subseteq\catgfc$ by Fact~\ref{fact0201},
we have $\finrescatfccot\subseteq\finrescatgfc$.
With this, part~\eqref{thma2} follows from Proposition~\ref{lem0210}.
Proposition~\ref{lem0303} and
Lemma~\ref{lem0304} justify part~\eqref{thma3}.
\qed
\end{para}

Here is the list of immediate consequences of
Theorem~\ref{thma} and~\cite[(1.12.10)]{hashimoto:abaem}.
For part~\eqref{cor0301a}, recall that 
$\operatorname{add}(\catx)$ is the subcategory of
all $R$-modules isomorphic to a direct summand of a finite direct sum of 
modules in $\catx$.

\begin{cor}  \label{cor0301}
Let $C$ be a semidualizing $R$-module and let $M\in\finrescatgfc$.
\begin{enumerate}[\quad\rm(a)]
\item \label{cor0301a}
If $\catx$ is an injective cogenerator for
$\catgfc$, then $\operatorname{add}(\catx)=\catfccot$.
\item \label{cor0301b}
There
exists an exact sequence $0\to Y \to X \to M \to 0$ with $X\in
\catgfc$ and $Y\in\finrescatfccot$.
\item \label{cor0301c}
There
exists an exact sequence $0\to M \to Y \to X \to 0$ with $X\in
\catgfc$ and $Y\in\finrescatfccot$.
\item \label{cor0301d}
The
following conditions are equivalent:
\begin{enumerate}[\quad\rm(i)]
\item $M\in \catgfc$;
\item $\ext_R^{\geq 1}(M,\finrescatfccott)=0$;
\item $\ext_R^1(M,\finrescatfccott)=0$;
\item $\ext_R^{\geq 1}(M,\catfccott)=0$.
\end{enumerate}
Thus, the surjection $X\to M$ from~\eqref{cor0301b} is a $\catgfcc$-precover of $M$.
\item \label{cor030e}
The
following conditions are equivalent:
\begin{enumerate}[\quad\rm(i)]
\item $M\in\finrescatfccot$;
\item $\ext_R^{\geq 1}(\catgfcc,M)=0$;
\item $\ext_R^1(\catgfcc,M)=0$;
\item $\sup\{i\geq 0\mid \ext^i_R(\catgfcc,M)\neq 0\}<\infty$ 
and $\ext^{\geq 1}_R(\catfccott,M)=0$.
\end{enumerate}
Thus, the injection $M\to Y$ from~\eqref{cor0301c} is a $\finrescatfccott$-preenvelope of $M$.
\item \label{cor0301f}
There are equalities
\begin{align*}
\catpd{\catgfcc}_R(M)
&=\sup\{i\geq 0\mid \ext^i_R(M,\finrescatfccott)\neq 0\} \\
&= \sup\{i\geq 0\mid \ext^i_R(M,\catfccott)\neq 0\}
\end{align*}
\item \label{cor0301g}
There is an inequality
$\catpd{\catgfcc}_R(M)\leq\catpd{\catfccott}_R(M)$
with equality when $\catpd{\catfccott}_R(M)<\infty$.
\item The category $\finrescatgfc$ 
is closed under extensions, kernels of epimorphisms and 
cokernels of monomorphisms.
\qed
\end{enumerate}
\end{cor}

For the next result recall that the triple
$(\catgfc,\finrescatfccot,\catfccot)$ is an AB-context
if it is a weak AB-context and such that
$\finrescatgfc=\catm$.

\begin{prop} \label{prop0301}
Assume that $\dim(R)$ is finite,
and let $C$ be a semidualizing $R$-module.
The triple
$(\catgfc,\finrescatfccot,\catfccot)$ is an AB-context
if and only if $C$ is dualizing for $R$.
\end{prop}

\begin{proof}
Assume first that 
$(\catgfc,\finrescatfccot,\catfccot)$ is an AB-context.
Recall that  every maximal ideal of
the trivial extension
$R\ltimes C$ is of the form $\m\ltimes C$ for some maximal ideal $\m\subset R$,
and there is an isomorphism
$(R\ltimes C)/(\m\ltimes C)\cong R/\m$.
With Fact~\ref{fact0201}, this yields the equality in 
the next sequence
\begin{align*}
\gfd_{(R\ltimes C)_{\m\ltimes C}}((R\ltimes C)_{\m\ltimes C}/(\m\ltimes C)_{\m\ltimes C})
&\leq\gfd_{R\ltimes C}((R\ltimes C)/(\m\ltimes C))\\
&=\gfcpd_R(R/\m) <\infty.
\end{align*}
The first inequality follows from~\cite[(5.1.3)]{christensen:gd},
and the finiteness is by assumption.
Using~\cite[(1.2.7),(1.4.9),(5.1.11)]{christensen:gd}
we deduce that the following ring is Gorenstein
$$
(R\ltimes C)_{\m\ltimes C}
\cong
R_{\m}\ltimes C_{\m}
$$
and so~\cite[(7)]{reiten:ctsgm}
implies that $C_{\m}$ is dualizing for $R_{\m}$.
(This also follows from~\cite[(8.1)]{christensen:scatac}
and~\cite[(3.1)]{holm:smarghd}.)
Since this is true for each maximal ideal of $R$ and $\dim(R)<\infty$, we conclude
that $C$ is dualizing for $R$ by~\cite[(5.8.2)]{hartshorne:rad}.

Conversely, assume that $C$ is dualizing for $R$.
Using Theorem~\ref{thma}, 
it suffices to show that each $R$-module $M$ has
$\gfcpd_R(M)<\infty$.
Since $C$ is dualizing, the trivial extension $R\ltimes C$
is Gorenstein by~\cite[(7)]{reiten:ctsgm}.
Also, we have $\dim(R\ltimes C)=\dim(R)<\infty$
as $\spec(R\ltimes C)$ is in bijection with $\spec(R)$.
Thus, in the next sequence
$$\gfcpd_R(M)=\gfd_{R\ltimes C}(M)<\infty$$
the finiteness 
is from~\cite[(12.3.1)]{enochs:rha}
and the equality is from Fact~\ref{fact0201}.
\end{proof}

To end this section, we prove a compliment to~\cite[(4.6)]{white:gpdrsm}
which establishes the existence of certain approximations.  
For this, we need the following preliminary
result which compares to Lemma~\ref{lem0304}.

\begin{lem} \label{lem1002}
If $C$ is a semidualizing $R$-module,
then there is an equality
$\catfc=\catgfc\cap\finrescatfc$.
\end{lem} 

\begin{proof}
The containment $\catfc\subseteq\catgfc\cap\finrescatfc$ is
from Definition~\ref{notation03}
and Fact~\ref{fact0201}. For the reverse containment, let
$M\in\catgfc\cap\finrescatfc$. Let $n\geq 1$ be an integer with $\fcpd_R(M)\leq n$.
We show by induction on $n$ that $M$ is $C$-flat. 

For the base case $n=1$, there is an exact sequence
\begin{equation} \label{lem1002a} \tag{$\dagger$}
0\to X_1\to X_0\to M\to 0
\end{equation}
with $X_1,X_0\in\catfc$.  Lemma~\ref{lem1001}
provides an exact sequence
\begin{equation} \label{lem1002b} \tag{$\ddagger$}
0\to X_1\to Y_1\to Y_2\to 0
\end{equation}
with $Y_1\in\catfccot$ and $Y_2\in\catfc$.
Consider the following pushout diagram
whose top row is~\eqref{lem1002a} and whose leftmost column is~\eqref{lem1002b}.
\begin{equation}\tag{$\ast$}
\begin{split} \label{lem1002c} 
\xymatrix{
& 0 \ar[d] & 0 \ar[d] \\
0\ar[r]  & X_1\ar[r] \ar[d] \ar@{}[rd]|>>{\lrcorner} & X_0\ar[r] \ar[d] & M\ar[r]\ar[d]_{\cong}  & 0 \\
0\ar[r] & Y_1\ar[r]\ar[d] & V\ar[r]\ar[d] & M\ar[r] & 0 \\
& Y_2\ar[r]^{\cong}\ar[d] & Y_2\ar[d] \\
& 0 & 0
}
\end{split}
\end{equation}
Since $M$ is in $\catgfc$ and $Y_1$ is in $\catfccot$, Lemma~\ref{extgfcvan}
implies $\ext^1_R(M,Y_1)=0$.  Hence, the middle row of~\eqref{lem1002c}
splits. The subcategory $\catfc$ is closed under extensions and summands 
by~\cite[Props.\ 5.1(a) and 5.2(a)]{holm:fear}. Hence, the middle column of~\eqref{lem1002c}
shows that $V\in\catfc$, so the fact that the middle row of~\eqref{lem1002c}
splits implies that $M\in\catfc$, as desired.

For the induction step, assume that $n\geq 2$.
Truncate a bounded $\catfcc$-resolution of $M$
to find an exact sequence 
$$0\to K\to Z\to M\to 0$$
such that $Z\in\catfc$ and $\fcpd_R(K)\leq n-1$.
By induction, we conclude that $K\in\catfc$.
Hence, the displayed sequence implies $\fcpd_R(M)\leq 1$,
and the base case implies that $M\in\catfc$.
\end{proof}

\begin{prop} \label{prop1001}
Let $C$ be a semidualizing $R$-module and assume that $\dim(R)$ is finite.
If $M\in\catgfc$, then there exists an exact sequence
$$0\to K\to X\to M\to 0$$
such that $K\in\catfc$ and $X\in\catgpc$.
\end{prop}

\begin{proof}
Since $M$ is in $\catgfc$ and $\dim(R)<\infty$, we know
that $\gpcpd_R(M)<\infty$ by~\cite[(3.3.c)]{sather:crct}. Hence,
from~\cite[(4.6)]{white:gpdrsm} there is an exact sequence
$$0\to K\to X\to M\to 0$$
with $K\in\finrescatpc$ and $X\in\catgpc$.
From~\cite[(3.3.a)]{sather:crct} we have $X\in\catgpc\subseteq\catgfc$.
Since $\catgfc$ is closed under kernels of epimorphisms
by Proposition~\ref{gfcprojres}, the displayed sequence
implies that $K\in\catgfc$. 
The containment $\catpc\subseteq\catfc$ implies
$K\in\finrescatpc\subseteq\finrescatfc$, and so
Lemma~\ref{lem1002} says $K\in\catfc$.
Thus, the displayed sequence has the desired properties.
\end{proof}

\section{Stability of Categories} \label{sec05}

This section contains our analysis of the categories
$\catg^n(\catfc)$ and $\catg^n(\catfccot)$; see Definition~\ref{defn0502}. 
We draw many of our conclusions from the known behavior 
for $\catg^n(\catic)$ using Pontryagin duals. This requires, however,
the use of the categories
$\cathc^n(\catfc)$ and $\cathc^n(\catfccot)$ as a bridge; see Definition~\ref{defn0501}. 

\begin{lem}  \label{lem0501}
Let $C$ be a semidualizing $R$-module, and let
$X$ be an $R$-complex.
If $X$ is $\hom_R(-,\catfccott)$-exact, then it is
$-\otimes_R\caticc$-exact.
\end{lem}

\begin{proof}
Let $N\in\catic$. From Lemmas~\ref{pdual}\eqref{pduald}
and~\ref{lem0201} we know that
the Pontryagin dual  $N^*$ is in $\catfccot$.
Hence, the following complex is exact
by assumption
$$\hom_R(X,N^*)
\cong\hom_R(X,\pdual{N})
\cong\pdual{X\otimes_RN}.
$$
As $\qmodz$ is faithfully injective over $\bbz$,
we conclude that $X\otimes_RN$ is exact,
and so $X$ is $-\otimes_R\caticc$-exact.
\end{proof}

Note that the hypotheses of the next lemma are satisfied whenever
$\catx\subseteq\catgfc$ by Fact~\ref{fact0201} and Lemma~\ref{extgfcvan}.

\begin{lem}  \label{lem0502}
Let $C$ be a semidualizing $R$-module and  $\catx$  a  subcategory
of $\catm$.  
\begin{enumerate}[\quad\rm(a)]
\item \label{lem0502a}
If $\tor^R_{\geq 1}(\catx,\caticc)=0$, 
then
$\tor^R_{\geq 1}(\cathc^n(\catx),\caticc)=0$  for each $n\geq 1$.
\item \label{lem0502b}
If $\catx\perp\catfccot$, 
then $\cathc^n(\catx)\perp\catfccot$ for each $n\geq 1$.
\end{enumerate}
\end{lem}

\begin{proof}
By induction on $n$, it suffices to prove the result for $n=1$.
We prove part~\eqref{lem0502a}. The proof of part~\eqref{lem0502b}
is similar.

Let $M\in\cathc(\catx)$ with $\catpcc\catfccott$-complete $\catx$-resolution
$X$.  
The complex $X$ is $-\otimes_R\caticc$-exact
by Lemma~\ref{lem0501}. 
Since we have assumed that
$\tor^R_{\geq 1}(\catx,\caticc)=0$, 
the desired conclusion follows from Lemma~\ref{lem0216}\eqref{lem0216c}
because $M\cong\ker(\partial^X_{-1})$.
\end{proof}

The converse of the next result is in Proposition~\ref{prop0502}.

\begin{lem}  \label{lem0503}
If $C$ is a semidualizing $R$-module and
$M\in\cathc(\catfc)$, then $M^*\in\catg(\catic)$.
\end{lem}

\begin{proof}
Let $X$ be a  $\catpcc\catfccott$-complete $\catfcc$-resolution of $M$.
Lemma~\ref{pdual}\eqref{pdualb}
implies that the complex $X^*=\pdual{X}$ is an exact
complex in $\catic$. Furthermore $M^*\cong\coker(\partial^{X^*}_{1})$.
Thus, it suffices to show that $X^*$ is
$\hom_R(\caticc,-)$-exact and $\hom_R(-,\caticc)$-exact.
Let $I$ be an injective $R$-module.

The second isomorphism in the following sequence
is Hom-evaluation~\cite[(0.3.b)]{christensen:apac} 
$$C\otimes_RX^*
\cong C\otimes_R\pdual{X}
\cong \pdual{\Hom_R(C,X)}.
$$
Since $\Hom_R(C,X)$ is exact by assumption,
we conclude that $C\otimes_RX^*\cong X^*\otimes_RC$ is also exact.
It follows that the following complexes are also exact
$$\hom_R(X^*\otimes_RC,I)
\cong\hom_R(X^*,\Hom_R(C,I))
$$
where the isomorphism is Hom-tensor adjointness.
Thus $X^*$ is $\hom_R(-,\caticc)$-exact.

Lemma~\ref{lem0501} implies that the complex
$\hom_R(C,I)\otimes_RX$ is exact. Hence, the following
complexes are also exact
\begin{align*}
\pdual{\hom_R(C,I)\otimes_RX}
&\cong\hom_R(\hom_R(C,I),\pdual{X})\\
&\cong\hom_R(\hom_R(C,I),X^*)
\end{align*}
and so $X^*$ is $\hom_R(\caticc,-)$-exact.
\end{proof}

The next result is a version of~\cite[(5.2)]{sather:sgc} for $\cathc(\catfc)$.

\begin{prop}  \label{prop0501}
If $C$ is a semidualizing $R$-module, then there is an equality
$\cathc(\catfc)=\catgfc\cap\catbc$.
\end{prop}

\begin{proof}
For the containment $\cathc(\catfc)\subseteq\catgfc\cap\catbc$,
let $M\in\cathc(\catfc)$, and
let $X$ be a $\catpcc\catfccott$-complete $\catfcc$-resolution of $M$.
Lemma~\ref{lem0501} implies that $X$ is 
$-\otimes_R\caticc$-exact, and so the 
sequence
$$0\to M\to X_{-1}\to X_{-2}\to\cdots$$
satisfies condition~\ref{fact0201}\eqref{fact0201a}.
Fact~\ref{fact0201} implies $\tor^R_{\geq 1}(\catfcc,\caticc)=0$
and so Lemma~\ref{lem0502}\eqref{lem0502a} provides
$\tor_{\geq 1}^R(M,\caticc)=0$.
From Fact~\ref{fact0201} we conclude $M\in\catgfc$.
Also, Lemma~\ref{lem0503} guarantees that
$M^*\in\catg(\catic)$,
and so $M^*\in\catac$
by Fact~\ref{fact0501}.
Thus, Fact~\ref{projac}
implies $M\in\catbc$.

For the reverse containment, let $M\in\catgfc\cap\catbc$,
and let $Y$ be a complete $\catff\catfcc$-resolution of $M$.
In particular, the complex
\begin{equation} \label{prop0501a} \tag{$\dagger$}
0\to M\to Y_{-1}\to Y_{-2}\to\cdots
\end{equation}
is an augmented $\catfcc$-coresolution of $M$
and is $-\otimes_R\caticc$-exact.
We claim that this complex is also $\hom_R(C,-)$-exact
and $\hom_R(-,\catfccott)$-exact.
For each $i\in\bbz$ set 
$M_i=\coker(\partial^Y_i)$.  This yields an isomorphism
$M\cong M_1$.
By assumption, we have $M,Y_i\in\catbc$ for each $i<0$,
and so $C\perp M$ and $C\perp Y_i$.
Thus, Lemma~\ref{lem0216}\eqref{lem0216b} implies
that the complex~\eqref{prop0501a}  is $\hom_R(C,-)$-exact.
From Lemma~\ref{lem0215} we conclude $M_i\in\catgfc$ for each $i$,
and so  $M_i\perp\catfccot$ by Lemma~\ref{extgfcvan}.
Lemma~\ref{lem0201} implies
$Y_i\perp \catfccot$ for each $i<0$, and so
Lemma~\ref{lem0216}\eqref{lem0216a} guarantees
that~\eqref{prop0501a}  is also $\hom_R(-,\catfccott)$-exact.

Because $M\in\catbc$,
Fact~\ref{projac} provides an augmented $\catpcc$-proper $\catpcc$-resolution
\begin{equation} \label{prop0501b} \tag{$\ddagger$}
\cdots\xra{\partial^Z_2} Z_1\xra{\partial^Z_1} Z_0\to M\to 0.
\end{equation}
Since each $Z_i\in\catpc\subseteq\catfc$, we have $Z_i\perp \catfccot$
by Lemma~\ref{lem0201}.
Since $M\perp\catfccot$, we see from 
Lemma~\ref{lem0216}\eqref{lem0216a} 
that~\eqref{prop0501b} is also $\hom_R(-,\catfccott)$-exact.

It follows that the complex obtained by splicing the sequences~\eqref{prop0501a}
and~\eqref{prop0501b} is a $\catpcc\catfccott$-complete $\catfcc$-resolution of $M$.
Thus $M\in\cathc(\catfc)$, as desired.
\end{proof}

Our next result contains the converse to Lemma~\ref{lem0503}.

\begin{prop}  \label{prop0502}
Let $C$ be a semidualizing $R$-module and $M$ an $R$-module.
Then $M\in\cathc(\catfc)$ if and only if $M^*\in\catg(\catic)$.
\end{prop}

\begin{proof}
One implication is in Lemma~\ref{lem0503}.
For the converse, assume that $M^*$ is in $\catg(\catic)=\catgic\cap\catac$; 
see Fact~\ref{fact0501}.
Fact~\ref{projac} and Lemma~\ref{pdual2} combine with
Proposition~\ref{prop0501} to
yield $M\in\catbc\cap\catgfc=\cathc(\catfc)$.
\end{proof}

The next three lemmata are for use in Theorem~\ref{thm0501}.

\begin{lem}  \label{lem0504}
If $C$ is a semidualizing $R$-module, then $\cathc^2(\catfc)\subseteq\catbc$.
\end{lem}

\begin{proof}
Let $M\in\cathc^2(\catfc)$ and let $X$ be a 
$\catpcc\catfccott$-complete $\cathc(\catfcc)$-resolution of $M$.
In particular, the complex $\hom_R(C,X)$ is exact.
Each module $X_i$ is in $\cathc(\catfc)\subseteq\catbc$
by Proposition~\ref{prop0501}, and so
$\ext^{\geq 1}_R(C,X_i)=0$ for each $i$.
Thus, Lemma~\ref{lem0216}\eqref{lem0216b}
implies that $\ext^{\geq 1}_R(C,M)=0$.
Also, since $M\cong\ker(\partial^X_{-1})$, the left-exactness of
$\hom_R(C,-)$ implies that
$\hom_R(C,M)\cong\ker(\partial^{\hom_R(C,X)}_{-1})$.

The natural evaluation map $C\otimes_R\hom_R(C,X_i)\to X_i$ is an isomorphism
for each $i$ because $X_i\in\catbc$,
and so we have $C\otimes_R\hom_R(C,X)\cong X$.  In particular,
the complex $\hom_R(C,X)$ is $-\otimes_R C$-exact.
As $\tor^R_{\geq 1}(C,\hom_R(C,X_i))=0$ for each $i$,
Lemma~\ref{lem0216}\eqref{lem0216c}
implies that $\tor^R_{\geq 1}(C,\hom_R(C,M))=0$.

Finally, each row in  the following diagram is exact
$$
\xymatrix{
C\otimes_R\hom_R(C,X_1)
\ar[r] \ar[d]_{\cong}
& C\otimes_R\hom_R(C,X_0)
\ar[r] \ar[d]_{\cong}
& C\otimes_R\hom_R(C,M)
\ar[r] \ar[d]
& 0 \\
X_1\ar[r] &X_0\ar[r] & M\ar[r] & 0
}$$
and the vertical arrows are the natural evaluation maps.
A diagram chase shows that the rightmost vertical arrow
is an isomorphism, and so $M\in\catbc$.
\end{proof}

\begin{lem}  \label{lem0505}
If $C$ is a semidualizing $R$-module, then $\catfccot$ is an injective
cogenerator for $\cathc(\catfc)$.
\end{lem}

\begin{proof}
The containment in the following sequence is from
Facts~\ref{projac} and~\ref{fact0201}
$$\catfccot\subseteq\catgfc\cap\catbc=\cathc(\catfc)$$
and the equality is from Proposition~\ref{prop0501}.
Lemma~\ref{extgfcvan} implies $\catgfc\perp\catfccot$.
Thus, the conditions $\cathc(\catfc)=\catgfc\cap\catbc
\subseteq\catgfc$ imply that we have $\cathc(\catfc)\perp\catfccot$.

Let $M\in\cathc(\catfc)\subseteq\catgfc$.  
Since $\catfccot$ is an injective cogenerator for $\catgfc$
by Proposition~\ref{lem0303},
there is an exact sequence
$$0\to M\to X\to M'\to 0$$
with $X\in\catfccot$ and $M'\in\catgfc$.
Since $M$ and $X$ are in $\catbc$,
Fact~\ref{projac} implies that $M'\in\catbc$.
That is $M'\in\catgfc\cap\catbc=\cathc(\catfc)$.
This extablishes the desired conclusion.
\end{proof}

\begin{lem}  \label{lem0506}
If $C$ is a semidualizing $R$-module, then $\cathc^2(\catfc)\subseteq\propcorescatfccot$.
\end{lem}

\begin{proof}
Lemma~\ref{lem0505} says that $\catfccot$ is an injective
cogenerator for $\cathc(\catfc)$.
By Lemma~\ref{lem0502}\eqref{lem0502b} we know that
$\cathc^2(\catfc)\perp\catfccot$.
Let $M\in\cathc^2(\catfc)$ and let $X$ be a
$\catpcc\catfccott$-complete $\cathc(\catfcc)$-resolution of $M$.
By definition, the complex
$$0\to M\to X_{-1}\to X_{-2}\to\cdots$$
is an augmented $\cathc(\catfcc)$-coresolution
that is $\catfcc$-proper
and therefore $\catfccott$-proper. Hence, Lemma~\ref{lem0602}
implies $M\in\propcorescatfccot$. 
\end{proof}

\begin{thm}  \label{thm0501}
For each semidualizing $R$-module $C$ and each integer $n\geq 1$,
there is an equality $\cathc^n(\catfc)=\catgfc\cap\catbc$.
\end{thm}

\begin{proof}
We first verify the equality $\cathc^2(\catfc)=\cathc(\catfc)$.
Remark~\ref{disc0501} implies $\cathc^2(\catfc)\supseteq\cathc(\catfc)$.  
For the reverse containment, let
$M\in\cathc^2(\catfc)$. 
Lemma~\ref{lem0201} implies $\catfc\perp\catfccot$,
and so $M\perp\catfccot$ by Lemma~\ref{lem0502}\eqref{lem0502b}.
From Lemma~\ref{lem0504} we have
$M\in\catbc$, and so Fact~\ref{projac}
provides an augmented $\catpcc$-proper $\catpcc$-resolution
\begin{equation} \label{thm0501b} \tag{$\ddagger$}
\cdots\xra{\partial^Z_2} Z_1\xra{\partial^Z_1} Z_0\to M\to 0.
\end{equation}
Each $Z_i\in\catpc\subseteq\catfc$, so we have $Z_i\perp \catfccot$
by Lemma~\ref{lem0201}.
We conclude from Lemma~\ref{lem0216}\eqref{lem0216a} 
that~\eqref{thm0501b} is $\hom_R(-,\catfccott)$-exact.

Lemma~\ref{lem0506} yields a $\catfccott$-proper augmented $\catfccott$-coresolution
\begin{equation} \label{thm0501a} \tag{$\dagger$}
0\to M\to Y_{-1}\to Y_{-2}\to\cdots.
\end{equation}
Since each $Y_i\in\catfccot\subseteq\catbc$ by Fact~\ref{projac},
we have $C\perp Y_i$ for each $i<0$, and similarly $C\perp M$.
Thus, Lemma~\ref{lem0216}\eqref{lem0216b} implies
that~\eqref{thm0501a}  is $\hom_R(C,-)$-exact.
It follows that the complex obtained by splicing the sequences~\eqref{thm0501b}
and~\eqref{thm0501a} is a $\catpcc\catfccott$-complete $\catfcc$-resolution of $M$.
Thus, we have $M\in\cathc(\catfc)$.

To complete the proof, use the previous two paragraphs and
argue by induction on $n$ to verify
the first equality in the next sequence
$$\cathc^n(\catfc)=\cathc(\catfc)=\catgfc\cap\catbc.$$
The second equality is from Proposition~\ref{prop0501}.
\end{proof}

Our next result contains Theorem~\ref{thmb}\eqref{thmb1} from the introduction.

\begin{cor} \label{cor0501}
If $C$ is a semidualizing $R$-module, then
$\catg^n(\catgfc\cap\catbc)=\catgfc\cap\catbc$
for each $n\geq 1$.
\end{cor}

\begin{proof}
In the next sequence, the containments are from 
Fact~\ref{fact0501} and Remark~\ref{disc0501}
\begin{align*}
\catgfc\cap\catbc
&\subseteq\catg^n(\catgfc\cap\catbc)
=\catg^n(\cathc(\catfc))\\
&\subseteq\cathc^n(\cathc(\catfc))
=\catgfc\cap\catbc
\end{align*}
and the equalities are by Proposition~\ref{prop0501}
and  Theorem~\ref{thm0501}.
\end{proof}

\begin{disc} \label{disc0801}
In light of Corollary~\ref{cor0501}, it is natural to ask
whether we have $\catg(\catfc)=\catgfc\cap\catbc$
for each semidualizing $R$-module $C$.  
While Remark~\ref{disc0501}
and Proposition~\ref{prop0501} imply that
$\catg(\catfc)\subseteq\catgfc\cap\catbc$,
we do not know whether the reverse containment holds.
\end{disc}

We now turn our attention to $\cathc^n(\catfccot)$ and $\catg^n(\catfccot)$.

\begin{prop} \label{prop0701}
Let $C$ be a semidualizing $R$-module and let $n\geq 1$.
\begin{enumerate}[\quad\rm(a)]
\item \label{prop0701a}
We have 
$\catgfc\cap\catbc\cap\catfc^{\perp}
\subseteq\cathc^n(\catfccot)\subseteq\catgfc\cap\catbc$.
\item \label{prop0701b}
If $\dim(R)<\infty$,
then $\catfc\perp\cathc^n(\catfccot)$.
\item \label{prop0701c}
If $\dim(R)<\infty$, then 
$\cathc^n(\catfccot)=\catgfc\cap\catbc\cap\catfc^{\perp}$.
\end{enumerate}
\end{prop}

\begin{proof}
\eqref{prop0701a}
For the first containment,
let $M\in \catgfc\cap\catbc\cap\catfc^{\perp}$.  
Since $M\in\catbc\cap\catfc^{\perp}$,
Lemma~\ref{lem0206}\eqref{lem0206c}
yields an augmented $\catfccott$-resolution
$$\cdots\to Z_1\to Z_0\to M\to 0$$
that is $\hom_R(C,-)$-exact;
the argument of Proposition~\ref{prop0501} shows that this resolution is
$\hom_R(-,\catfccott)$-exact.
Because $M$ is in $\catgfc$, Proposition~\ref{lem0303}
provides an augmented $\catfccott$-coresolution
$$0\to M\to Y_{-1}\to Y_{-2}\to\cdots$$
that is $\Hom_R(-,\catfccott)$-exact.
Since $M\in\catbc$, the proof of Proposition~\ref{prop0501}
shows that this coresolution is also $\Hom_R(C,-)$-exact.
Splicing these resolutions yields a $\catpcc\catfccott$-complete $\catfccott$-resolution of $M$,
and so $M\in\cathc(\catfccot)\subseteq\cathc^n(\catfccot)$.

The second containment follows from the next sequence
$$\cathc^n(\catfccot)\subseteq\cathc^n(\catfc)=\catgfc\cap\catbc$$
wherein the containment is by definition, and the equality is
by Theorem~\ref{thm0501}.

\eqref{prop0701b}
Assume $d=\dim(R)<\infty$. 
A result of Gruson and 
Raynaud~\cite[Seconde Partie, Thm.~(3.2.6)]{raynaud:cpptpm} and 
Jensen~\cite[Prop.~6]{jensen:vl} implies $\pd_R(F)\leq d<\infty$ 
for each flat $R$-module $F$.

We prove the result for all $n\geq 0$ by induction on $n$.
The base case $n=0$ follows from Lemma~\ref{lem0201}.
Assume $n\geq 1$ and that $\catfc\perp\cathc^{n-1}(\catfccot)$.
Let $M\in\cathc^n(\catfccot)$,
and let $X$ be a $\catpcc\catfccott$-complete $\cathc^{n-1}(\catfccott)$-resolution of $M$.
For each $i$ set $M_i=\im(\partial^X_i)$.
This yields an isomorphism $M\cong M_0$ and, for each $i$, an exact sequence
$$0\to M_{i+1}\to X_i\to M_i\to 0.$$
Note that $M_i,X_i\in\catbc$ by part~\eqref{prop0701a}.
Let $F\otimes_RC\in\catfc$ and let $t\geq 1$.
Since $\catfc\perp X_i$ for each $i$, a standard dimension-shifting argument
yields the first isomorphism in the next sequence
\begin{align*}
\ext^t_R(F\otimes_RC,M)
&\cong\ext^{t+d}_R(F\otimes_RC,M_d)
\cong\ext^{t+d}_R(F,\hom_R(C,M_d))=0.
\end{align*}
The second isomorphism is a form of Hom-tensor adjointness
using the fact that $F$ is flat with the Bass class condition
$\ext^{\geq 1}_R(C,M_d)=0$.
The vanishing follows from the inequality $\pd_R(F)\leq d$.

\eqref{prop0701c}
This follows from parts~\eqref{prop0701a}
and~\eqref{prop0701b}.
\end{proof}

\begin{lem} \label{lem0801}
Let $C$ be a semidualzing $R$-module and assume
$\dim(R)<\infty$.  If $M\in\catfc$, then $\catfccott\text{-}\id_R(M)\leq\dim(R)<\infty$.
\end{lem}

\begin{proof}
Let $F$ be a flat $R$-module such that $M\cong F\otimes_R C$.
Since $d=\dim(R)$ is finite, the flat module $F$ has
an $\catfcott$-coresolution
$X$ such that $X_i=0$ for all $i<-d$;
see \cite[(8.5.12)]{enochs:rha}.
Since $M\in\catac$ and each $X_i\in\catac$,
it follows readily that the complex $X\otimes_R F$
is an $\catfccott$-coresolution of $M$ of length
at most $d$, as desired.
\end{proof}

Our final result contains Theorem~\ref{thmb}\eqref{thmb2} from the introduction.

\begin{thm} \label{prop0702}
Let $C$ be a semidualzing $R$-module and assume
$\dim(R)<\infty$.  Then
$\catg^n(\catfccot)=\catgfc\cap\catbc\cap\catfc^{\perp}$
for each $n\geq 1$,
and $\catfccot$ is an injective cogenerator and a 
projective generator for
$\catgfc\cap\catbc\cap\catfc^{\perp}$.
\end{thm}

\begin{proof}
We first show $\catg(\catfccot)\supseteq\cathc(\catfccot)$.
Let $M\in\cathc(\catfccot)$ and let $X$ be a $\catpcc\catfccott$-complete 
$\catfccott$-resolution of $M$.
To show that $M$ is in $\catg(\catfccot)$, it suffices to show
that $X$ is $\hom_R(\catfccott,-)$-exact, since it is
$\hom_R(-,\catfccott)$-exact by definition.
For each $i$, set $M_i=\im(\partial^X_i)\in\cathc(\catfccot)$.
Lemma~\ref{lem0201} and
Proposition~\ref{prop0701}\eqref{prop0701b}
imply $\catfc\perp X_i$ and $\catfc\perp M_i$ for all $i$.
Hence, Lemma~\ref{lem0216}\eqref{lem0216b}
implies that $X$ is $\hom_R(\catfcc,-)$-exact, and so
$X$ is $\hom_R(\catfccott,-)$-exact. 

We next show $\catg(\catfccot)\subseteq\cathc(\catfccot)$.
Let $N\in\catg(\catfccot)$ and let $Y$ be a complete 
$\catfccott$-resolution of $N$.  
We will show that $Y$ is $\hom_R(\catfcc,-)$-exact;
the containment $\catpc\subseteq \catfc$ will then imply that
$Y$ is $\hom_R(\catpcc,-)$-exact.
Since $Y$ is 
$\hom_R(-,\catfccott)$-exact by definition,
we will then conclude  that $N$ is in $\cathc(\catfccot)$.
We have $\catfc\perp Y_i$ for each $i$ by Lemma~\ref{lem0201},
and so $\catfccot\perp Y_i$. Since $Y$ is $\hom_R(\catfccott,-)$-exact,
Lemma~\ref{lem0216}\eqref{lem0216b} implies $\catfccot\perp M$.
From Lemma~\ref{gencat01} we conclude that $\fincorescatfccot\perp M$.
Since $\dim(R)<\infty$, Lemma~\ref{lem0801} implies that
$\catfc\subseteq \fincorescatfccot$ and so $\catfc\perp M$.
With the condition $\catfc\perp Y_i$ from above, this implies that
$Y$ is $\hom_R(\catfcc,-)$-exact by Lemma~\ref{lem0216}\eqref{lem0216b}.

The above paragraphs yield the second equality in the next sequence
\begin{align*}
\catg^n(\catfccot)
&=\catg(\catfccot)=\cathc(\catfccot) =\catgfc\cap\catbc\cap\catfc^{\perp}.
\end{align*}
The first equality is from ~\cite[(4.10)]{sather:sgc}
since Lemma~\ref{lem0201}
implies $\catfccot\perp\catfccot$, and the third
equality is  from Proposition~\ref{prop0701}\eqref{prop0701c}.
The final conclusion follows from~\cite[(4.7)]{sather:sgc}.
\end{proof}

\providecommand{\bysame}{\leavevmode\hbox to3em{\hrulefill}\thinspace}
\providecommand{\MR}{\relax\ifhmode\unskip\space\fi MR }
\providecommand{\MRhref}[2]{%
  \href{http://www.ams.org/mathscinet-getitem?mr=#1}{#2}
}
\providecommand{\href}[2]{#2}

\end{document}